\documentclass[12pt,a4paper]{article}
\usepackage[utf8]{inputenc}
\usepackage{amssymb,amsmath,mathtools,amsthm,bbm}
\usepackage{fullpage}
\usepackage{tikz}
\usepackage{pgfplots}
\pgfplotsset{compat=1.16}
\usepackage{thmtools,thm-restate}
\usepackage{subcaption}

\usepackage[numbers]{natbib}
\usepackage{mleftright}
\mleftright
\usepackage{hyperref}
\hypersetup{hidelinks,
 pdfstartview=FitH,
 colorlinks=true,
 linkcolor=blue,
 citecolor=blue,
 urlcolor=blue}
\newcommand\emaillink[1]{\href{mailto:#1}{\nolinkurl{#1}}}

\newtheorem{theorem}{Theorem}[section]
\newtheorem{lemma}[theorem]{Lemma}
\newtheorem{prop}[theorem]{Proposition}

\newtheorem{claim}[theorem]{Claim}

\newtheorem{conjecture}[theorem]{Conjecture}

\theoremstyle{definition}

\numberwithin{equation}{section}

\newcommand\floor[1]{\left\lfloor #1 \right\rfloor}

\DeclareMathOperator{\sgn}{sgn}

\DeclareMathOperator{\Geom}{Geom}
\DeclareMathOperator{\Exp}{Exp}
\DeclareMathOperator{\Unif}{Uniform}

\newcommand\eps{\varepsilon}
\newcommand\N{\mathbb{N}}
\newcommand\Z{\mathbb{Z}}
\newcommand\R{\mathbb{R}}
\newcommand\E{\mathbb{E}}
\newcommand\Prb{\mathbb{P}}
\newcommand\one{\mathbbm{1}}
\newcommand\bigmid{\mathrel{\Big|}}

\newcommand\cY{\mathcal{Y}}
\newcommand\cF{\mathcal{F}}

\newcommand\event{E}
\newcommand\lw{Y} 
\newcommand\uw{\hat{Y}} 
\newcommand\lws{L} 
\newcommand\uws{B} 
\newcommand\lb{Y^{br}} 
\newcommand\ub{\hat{Y}^{br}} 
\newcommand\lwa{A} 
\newcommand\uwa{\hat{A}} 
\newcommand\lba{A^{br}} 
\newcommand\lwae{X} 
\newcommand\lbae{X^{br}} 
\newcommand\pbae{\tilde{X}^{br}} 
\newcommand\lwac{S} 
\newcommand\lbac{S^{br}} 
\newcommand\pbac{\tilde{S}^{br}} 
\newcommand\adjlw{\bar{Y}} 
\newcommand\adjlwa{\bar{A}} 
\newcommand\lmgf{f} 
\newcommand\probareazero[1]{\rho_{#1}}
\newcommand\probareapositive[1]{p_{#1}}
\newcommand\probbridgeareapositive[1]{q_{#1}}
\newcommand\cdeg{c_{\textup{deg}}}

\definecolor{applegreen}{rgb}{0.55, 0.71, 0.0}

\title{Counting graphic sequences via integrated random walks}

\author{Paul Balister\footnote{Mathematical Institute, University of Oxford, Oxford OX2\thinspace6GG, UK.
Email: \emaillink{paul.balister@maths.ox.ac.uk}. Partially supported by EPSRC grant EP/W015404/1.}\quad
Serte Donderwinkel\footnote{Bernoulli Institute, University of Groningen, Nijenborgh~9, 9747 AG Groningen, Netherlands and CogniGron (Groningen Cognitive Systems and Materials Center), University of Groningen, Nijenborgh~4, 9747 AG Groningen, Netherlands, Email: \emaillink{s.a.donderwinkel@rug.nl}.
}\quad
Carla Groenland\footnote{Delft Institute of Applied Mathematics, TU Delft, The Netherlands. Email:
\emaillink{c.e.groenland@tudelft.nl}. Part of her work was supported by the Marie Skłodowska-Curie grant
GRAPHCOSY (number 101063180) hosted at Utrecht University.}\\
Tom Johnston\footnote{School of Mathematics, University of Bristol, Bristol, BS8\thinspace1UG, UK
and Heilbronn Institute for Mathematical Research, Bristol, UK. Email: \emaillink{tom.johnston@bristol.ac.uk}.}\quad
Alex Scott\footnote{Mathematical Institute, University of Oxford, Oxford OX2\thinspace6GG, UK.
Email: \emaillink{scott@maths.ox.ac.uk}. Research supported by EPSRC grant EP/X013642/1.}}

\date{\today}

\begin{document}

\maketitle

\begin{abstract}
Given an integer $n$, let $G(n)$ be the number of integer sequences $n-1\ge d_1\ge d_2\ge\dotsb\ge d_n\ge 0$ that are the degree sequence of some graph. We show that $G(n)=(c+o(1))4^n/n^{3/4}$ for some constant $c>0$, improving both the previously best upper and lower bounds by a factor of $n^{1/4}$ (up to polylog-factors).   

Additionally, we answer a question of Royle, extend the values of $n$ for which $G(n)$ is known exactly from $n \leq 290$ to $n \leq 1651$ and determine the asymptotic probability that the integral of a (lazy) simple symmetric random walk bridge remains non-negative.
\end{abstract}

\section{Introduction}
Given a graph $G$ and a vertex $v\in V(G)$, the \emph{degree} of $v$ is the number of edges incident to~$v$,
and the \emph{degree sequence} of $G$ is the non-increasing sequence of its vertex degrees. We consider the
following very natural question: over all graphs on $n$ vertices, how many different degree sequences are there? 

Since the degree of a vertex is at most $n-1$ and at least~$0$, a simple upper bound follows by bounding the
number of integer sequences $n-1\ge d_1\ge d_2\ge\dotsb\ge d_n\ge 0$. A `stars-and-bars' argument
shows that there are $\binom{2n-1}{n-1}=\Theta(4^n/\sqrt{n})$ such sequences, but not all of them
are degree sequences of graphs. 
Sequences which are the degree sequence of some graph are called \emph{graphic sequences}. A famous result of
Erd\H{o}s and Gallai~\cite{ErdosGallai} provides necessary and sufficient conditions for a sequence to be graphical
and various other characterisations are known~\cite{Hakimi,Havel}.

Let $G(n)$ be the number of graphic sequences of length $n$ (or equivalently the number of degree
sequences across graphs on $n$ vertices). The best known bounds on $G(n)$ were given by
Burns~\cite{Burns2007TheNO} who showed that 
\[
 \frac{c_1 4^n}{n}\le G(n)\le \frac{4^n}{\sqrt{n}\log^{c_2}n}
\]
for some constants $c_1,c_2>0$ and for all $n\in \N$. To the best of our knowledge, these were the best known
asymptotics before our work, and this has been explicitly mentioned as an open problem in several
computational papers (e.g.\ \cite{LuBressan11,Wang19}). 

Our main result pinpoints the asymptotics for $G(n)$.

\begin{theorem}\label{thm:main}
The number of graphic sequences of length $n$ is $G(n)=(\cdeg+o(1))4^n/n^{3/4}$,
 where $\cdeg>0$ is a constant.
\end{theorem}
This also answers in the affirmative a question of Royle~\cite{royle} who asked\footnote{\label{footnote:royle}Royle
actually asked whether $G'(n)/G'(n-1)\to 4$ where $G'(n)$ is the number of degree sequences
of graphs without isolated vertices. However $G'(n)=G(n)-G(n-1)$,
so this question is equivalent to the question of whether $G(n)/G(n-1)$ tends to~4.}
whether the ratio $G(n)/G(n-1)$ tends to~4.

\paragraph{The value of $\cdeg$ and a connection to random walks.}
We can express the value of the constant~$\cdeg$ in terms of the hitting probabilities of a particular random walk,
which arises from our proof strategy. In particular, computational estimates give 
$\cdeg\approx 0.099094$ (see Section~\ref{subsec:rho}).

We will prove Theorem~\ref{thm:main} by viewing a sequence $n-1\ge d_1\ge \dotsb \ge d_n\ge 0$
as a path on a grid from $(0,n)$ to $(n-1,0)$, as depicted in Figure~\ref{fig:reformulation}(a).
We then count the number of graphic sequences of length $n$ by sampling such a path uniformly at random,
and computing the probability that the path satisfies the conditions given by 
Erd\H{o}s and Gallai for a sequence to be graphical.

To be more precise, a sequence is graphical if and only if the sum of the degrees is even and it satisfies the dominating condition given in (\ref{eq:DC}). Our aim is to show that the  probability that a random sequence satisfies the dominating condition is asymptotically $4\sqrt{\pi}\cdeg n^{-1/4}$. Then, we show that, asymptotically, half of the sequences that satisfy the dominating condition have even sum, and the result then follows from the fact that we considered $\binom{2n-1}{n-1}\approx(2\sqrt{\pi})^{-1}4^n n^{-1/2}$ sequences in total.

Via a number of reformulations (see Section \ref{sec:reformulation}), the probability that a uniformly random sequence satisfies the dominating condition turns out to be the probability that a particular integrated random walk bridge stays non-negative.
Let $Y=(\lw_k)_{k\geq 0}$ be a random walk that has increments that take the value $1$ with probability $\frac14$, the value $-1$ with probability $\frac14$ and the value $0$ otherwise, so that $Y$ is a \emph{lazy simple symmetric random walk}. Let  $\lwa_k=\sum_{j=1}^k\lw_j$ be its area process. Our probability of interest is \emph{the probability that $\lwa_1,\dots,\lwa_{n-1}\geq 0$, conditional on the event that $\lw_{n-1}\in \{0,-1\}$}. To introduce its asymptotic value, we need some additional notation.
Let $\zeta_1=\inf\{k\geq 1 ~:~ \lw_k=0, \lwa_k\le 0\}$ be the first visit of $\lw$ to $0$ at which $(\lwa_k)_{k=1}^\infty$ hits $(-\infty,0]$ and let $\rho = \Prb(\lwa_{\zeta_1}=0)$.
We prove the following result about lazy simple symmetric random walk bridges which may be of independent interest.
\begin{prop}\label{prop:areanonneg}
 We have that
 \[
  n^{1/4}\Prb(\lwa_1,\dots,\lwa_{n}\ge 0\mid \lw_n=0) \to \frac{\Gamma(3/4)}{\sqrt{2\pi(1-\rho)}}
 \]
 as $n\to\infty$.
\end{prop}
The probability that a random process does not take negative values is also called the \emph{persistence probability}.
The persistence probability of integrated random processes was first studied by Sina\u{\i} \cite{Sinai1992} in 1992, who showed that the persistence probability of an $n$-step simple symmetric random walk (SSRW) is $\Theta(n^{-1/4})$. The sharp asymptotics (including the constant) follow from a result by Vysotsky \cite[Theorem 1]{Vysotsky2014}. His work on random walk bridges implies that the persistence probability of an $n$-step SSRW bridge is $\Theta(n^{-1/4})$ \cite[Proposition 1]{Vysotsky2014}. The sharp asymptotics for SSRW bridges are a natural next question, which we answer in Proposition \ref{prop:persistenceSSRW} for the SSRW bridge and in Proposition \ref{prop:areanonneg} for the lazy variant. 

We use Proposition \ref{prop:areanonneg} to show that $G(n)=(\cdeg+o(1))4^n/n^{3/4}$ for
\[
 \cdeg = \frac{\Gamma(3/4)}{4\pi\sqrt{2(1-\rho)}}.
\]
The probability generating function of the area of the first excursion of $\lw$ away from $0$ satisfies a recursive equation which allows us to estimate that
$\rho$ is approximately $0.5158026$, and plugging this into the equation gives $\cdeg\approx 0.099094$. 

A more direct expression for $\rho$ is given in a follow-up work by Bassan, the second author and Kolesnik \cite{Bassan24}. There it is shown that 
$\rho = 1 - e^{-\xi}$ where \[\xi=\sum_{n=1}^\infty \frac{2}{n^2 2^{2n}} \sum_{d\mid n} \binom{2d-1}{d} \phi(n/d)\]
with $\phi$ Euler's totient function. This alternative form can be used to confirm the approximations of $\rho$ and $\cdeg$ that we give above and in Section \ref{subsec:rho}.

\paragraph{Related counting problems.}
Much more is known about related counting problems, such as the number of graphs
with a given degree sequence \cite{BarvinokH13,McKay1985,McKay_Wormald90,McKay_Wormald91,Wormald19} 
and a variant of our problem where the sequence does not
need to be non-increasing (e.g.\ \cite{Stanley}).

The number $T(n)$ of out-degree sequences for $n$-vertex tournaments, also called score sequences,
has received particular interest, and the problem of determining $T(n)$ can be traced back to MacMahon in 1920 \cite{macmahon1920american}.  Following work of Moser \cite{Moser}, Erd\H{o}s and Moser (see \cite{moon1968topics}), and Kleitman \cite{kleitman1970number},
it was shown that $T(n)=\Theta(4^n/n^{5/2})$ by Winston and Kleitman~\cite{KleitmanWinston} (lower bound) and Kim and Pittel~\cite{KimPittel} (upper bound).
Recently, Kolesnik \cite{Kolesnik22} determined the exact asymptotics,
showing that there is a constant $c\approx 0.392$ such that $T(n)=(c+o(1))4^n/n^{5/2}$. 

Another well-studied variant is the fraction $p(N)$ of partitions of an integer $N$ that are graphical.
This corresponds to the variant of our problem where we fix the number of edges of the graph,
rather than the number of vertices.
In 1982, Wilf conjectured that $p(N)\to 0$ as $N\to \infty$.
Pittel \cite{Pittel99} resolved this problem in the affirmative using a Kolmogorov zero-one law,
and Erd\H{o}s and Richmond \cite{ErdosRichmond93} showed a lower bound of $p(N)\ge \pi/\sqrt{6N}$
for sufficiently large even~$N$. The best-known upper bound is $p(N)=O(N^{-\alpha})$ for $\alpha\approx 0.003$
from Melczer, Michelen and Mukherjee~\cite{MelczerMichelenMukherjee20}. 

\paragraph{Exact enumeration.}
We also give an improved algorithm for the exact enumeration of graphic sequences and calculate
the number of graphic sequences of length $n$ for all $n$ up to~1651. 

There is previous work on enumerating graphic sequences \cite{ivanyi2011,Ivanyi2012,Ivanyi2013,ruskey1994,Wang2019b},
and the numbers were known up to $n=118$~\cite{Wang19} as OEIS sequence \href{https://oeis.org/A004251}{A004251},
although the numbers $G'(n)$ of zero-free graphic sequences were known up to $n=290$~\cite{Wang2019b},
and these imply the values of $G(n)$ up to the same number by the observation in footnote~\ref{footnote:royle}.

\paragraph{Paper overview.}
In Section~\ref{sec:prelim}, we provide the reformulation of our problem in terms of integrated random walks.
In Section~\ref{sec:proof}, we first show that $G(n) = \Theta(4^n/n^{3/4})$, and then prove
Theorem~\ref{thm:main} and Proposition~\ref{prop:areanonneg}, up to technical lemmas that we postpone to Section~\ref{sec:technical}.
In Section~\ref{sec:comp} we introduce an improved algorithm for computing $G(n)$ exactly
for small $n$ and discuss computational results such as the approximation of $\rho$. We conclude with some open problems in
Section~\ref{sec:conclusion}. An overview of the notation used throughout the paper is given in Appendix \ref{sec:notation_overview}.

\section{Reformulation and notation}
\label{sec:prelim}
In order to prove our results, we need a suitable criterion for when a sequence of non-negative integers
is the degree sequence of a simple graph. 
For a given sequence of non-negative integers $d_1\ge d_2\ge\dotsb\ge d_n\ge 0$,
let $d_i'$ be the number of $j$ with $d_j\ge i$ and set
\begin{align*}
 s_i &= (n-1)-d_i,\\
 s_i'&= n-d_i'.
\end{align*}
Let $\ell$ be the largest $j$ such that $d_j\ge j$. Using these definitions we can now
give the version of the Erd\H{o}s--Gallai Theorem that we will use.
\begin{theorem}[Variant of Erd\H{o}s--Gallai]
\label{thm:EG}
 A sequence of integers $d_1\ge d_2\ge\dotsb\ge d_n\ge 0$ is the degree sequence of a
 simple graph if and only if the sum $\sum_{i=1}^n d_i$ is even and for all\/ $k\le\ell$
 \begin{equation}\label{eq:DC}
  \sum_{i=1}^k s_i \ge \sum_{i=1}^k s_i'.
 \end{equation}
\end{theorem}
We remark that this is just one of many similar characterizations of graphic sequences and that this form follows from the classical statement of the Erd\H{o}s--Gallai Theorem by rearranging terms and observing that only the first $\ell$ conditions need be checked (see e.g. \cite{hasselbarth1984verzweigtheit,nash-williamsValency,rousseau1995note,sierksma1991seven}).

We will call \eqref{eq:DC} the \emph{dominating condition}, so that a sequence of non-negative integers is a graphic
sequence if it satisfies the dominating condition and its sum is even. 
As might be expected, it turns out that about half of the sequences that satisfy the
dominating condition have an even sum although, surprisingly, the exact proportion
seems to converge to $1/2$ quite slowly as $n\to\infty$ (see Section~\ref{subsec:eo}).
We will first focus on counting the number of sequences that satisfy the dominating condition, and we will handle the parity condition in Section \ref{subsec:half} (and briefly in the proof of Proposition \ref{prop:firstasymptotics}).

\subsection{It's a walk!}\label{sec:reformulation}
We now describe a way of associating a random walk to a uniformly random graphic sequence
in such a way that we can easily check whether the sequence is graphic using only the walk. 
\begin{lemma}\label{lem:to_walk}
 Let $(\lw_i)_{i\ge1}$ be a lazy simple symmetric random walk and let $\lwa_k=\sum_{i=1}^k\lw_i$.
 Then the probability that a uniformly random sequence $n-1\ge d_1\ge\dotsb\ge d_n\ge 0$
 satisfies the dominating condition \eqref{eq:DC} is equal to
 \begin{equation}\label{eq:probDC}
  \Prb\big(\lwa_1,\dots,\lwa_{n-1}\ge 0  \mid \lw_{n-1}\in\{0,-1\}\big).
 \end{equation}
 The probability that it is graphic is 
  \begin{equation}\label{eq:probgraphic}
  \Prb\big(\lwa_1,\dots,\lwa_{n-1}\ge 0, \lwa_{n-1}\in 2\N \mid \lw_{n-1}\in\{0,-1\}\big).
 \end{equation}
\end{lemma}
This lemma already provides a heuristic for $G(n)=\Theta(4^n/n^{3/4})$. Note that there are
$\binom{2n-1}{n}=\Theta(4^n/\sqrt{n})$ sequences $n-1\ge d_1\ge\dotsb\ge d_n\ge 0$, so we want
\eqref{eq:probDC} to be $\Theta(n^{-1/4})$. Firstly, note that we only need to check that
$\lwa_k\ge 0$ for $k$ such that $\lw_k=0$ (and for $k=n-1$), because
$\lwa_k$ is monotone on excursions of $\lw$ away from~$0$. Up to time~$n$,
the walk $\lw$ visits zero $\Theta(n^{1/2})$ times (in probability) so the area process restricted to
times $k$ when $\lw_k=0$ is a random process with $\Theta(n^{1/2})$ steps. If this process
were a random walk, then a result from Feller \cite[Theorem XII.7.1a]{Feller1971} would tell
us that the probability that it stays non-negative is $\Theta((n^{1/2})^{-1/2})=\Theta(n^{-1/4})$. We also require that the total area $A_{n-1}$ is even, but intuitively this should happen with probability roughly $1/2$. Of course, the lengths (and hence the areas) of the excursions are not independent and we cannot apply Feller's result, but this heuristic does at least suggest the right order for~\eqref{eq:probDC}.

\begin{figure}
\centering
\subcaptionbox{Step 1: View a sequence as a lattice path.\label{subfig:step1}}[0.4\textwidth]{%
 \tikzset{vertex/.style={draw,shape=circle,fill=black,minimum size=2pt,inner sep=0}}
\begin{tikzpicture}[scale=0.37]

    \draw (0,13)--(0,0)--(14,0); 
    \foreach \x in {1,...,14}{\draw[very thin, gray] (\x,0)--(\x,13);} 
    \foreach \y in {1,...,13}{\draw[very thin, gray] (0,\y)--(14,\y);}

    \draw (1,0)--(1,11) (2,0)--(2,10) (3,0)--(3,10) (4,0)--(4,9) (5,0)--(5,7)
    (6,0)--(6,5) (7,0)--(7,5) (8,0)--(8,4) (9,0)--(9,4) (10,0)--(10,4) (11,0)--(11,4)
    (12,0)--(12,4) (13,0)--(13,3); 

    \draw[red,thick] (0,13)--(1,13)--(1,11)--(2,11)--(2,10)--(4,10)--(4,9)--(5,9)--(5,7)
    --(6,7)--(6,6);
    \draw[green!70!black,thick](6,6)--(6,5)--(8,5)--(8,4)--(13,4)--(13,3)--(14,3)--(14,0); 

    \node[vertex] at (0,13){}; 
    \node[left] at (0, 13) {$\scriptstyle n-1$};
    \node[vertex] at (14,0){};
    \node[below] at (14, 0) {$\scriptstyle n$};

    \node at (.5,-.5){$\scriptstyle d_1$};\node at (1.5,-.5){$\scriptstyle d_2$};\node at (3,-.5){$\cdots$};
    \draw[<->] (3.5,10.1)--(3.5,12.9);\node[fill=white, right, inner sep = 1pt] at (3.8,11.5){$\scriptstyle s_4$}; 
    \draw[<->] (13.1,3.5)--(13.9,3.5);\node[fill=white, above, inner sep = 0pt] at (13.5,3.8){$\scriptstyle s'_4$}; 
    \draw[dashed](0,0)--(13, 13); 
    \node[vertex] at (6,6){};
    \node[fill=white, right, inner sep = 1pt] at (6.3,5.9) {$\scriptstyle (\ell,\ell)$};
\end{tikzpicture}%
}\hfill%
\subcaptionbox{Step 2: Cut the path in half and reflect the latter half.\label{subfig:step2}}[0.4\textwidth]{%
  \tikzset{vertex/.style={draw,shape=circle,fill=black,minimum size=2pt,inner sep=0}}
\centering
\begin{tikzpicture}[scale=0.37]

    \draw (0,13)--(0,0)--(14,0); 
    \foreach \x in {1,...,14}{\draw[very thin, gray] (\x,0)--(\x,13);} 
    \foreach \y in {1,...,13}{\draw[very thin, gray] (0,\y)--(14,\y);}

    \draw[dashed](0,0)--(13, 13); 

    \node at (.5,-.5){$\scriptstyle d_1$};\node at (1.5,-.5){$\scriptstyle d_2$};\node at (3,-.5){$\cdots$};

    \draw[red,thick] (0,13)--(1,13)--(1,11)--(2,11)--(2,10)--(4,10)--(4,9)--(5,9)--(5,7)--(6,7)--(6,6); 
    \draw[red,thick,dotted] (6,6)--(6,5);
    \draw[green!70!black,thick] (0,13)--(3,13)--(3,12)--(4,12)--(4,7)--(5,7)--(5,6)--(5,5)--(6,5); 

    \draw (6,5)--(8,5)--(8,4)--(13,4)--(13,3)--(14,3)--(14,0); 

    \node[vertex] at (0,13){}; 
    \node[left] at (0, 13) {$\scriptstyle n-1$};
    \node[vertex] at (14,0){};
    \node[below] at (14, 0) {$\scriptstyle n$};

    \node[vertex] at (6,6){};

    \node at (2,9.5) [inner sep = 1pt, below left, red, fill = white] {$W$};

    \node at (0.5,13) [below = 1pt, font=\tiny,red, inner sep = 1pt] {$1$};
    \node at (1.5,11) [font=\tiny,red, above = 1pt, inner sep = 1pt] {$4$};
    \node at (2.5,10) [font=\tiny,red, below = 1pt, inner sep = 1pt] {$6$};
    \node at (3.5,10) [font=\tiny,red, above = 1pt, inner sep = 1pt] {$7$};
    \node at (4.5,9) [font=\tiny,red, below = 1pt, inner sep = 1pt] {$9$};
    \node at (5.5,7) [font=\tiny,red, above = 1pt, inner sep = 1pt] {$12$};
    \node at (1,12.5) [font=\tiny,red,right = 1pt, inner sep = 1pt] {$2$};
    \node at (1,11.5) [font=\tiny,red,left = 1pt, inner sep = 1pt] {$3$};
    \node at (2,10.5) [font=\tiny,red,left = 1pt, inner sep = 1pt] {$5$};
    \node at (4,9.5) [font=\tiny,red,left = 1pt, inner sep = 1pt] {$8$};
    \node at (5,8.5) [font=\tiny,red,right = 1pt, inner sep = 1pt] {$10$};
    \node at (5,7.5) [font=\tiny,red,left = 1pt, inner sep = 1pt] {$11$};
    \node at (6,6.5) [font=\tiny,red,right = 1pt, fill = white, inner sep = 1pt] {$13$};
    \node at (6,5.4) [font=\tiny,red,right = 1pt, inner sep = 1pt] {$14$};

    \node at (4.5,12) [green!70!black, fill=white, inner sep = 1pt, above right] {$W'$};
    \node at (0.5,13) [font=\tiny,green!70!black,above = 1pt, inner sep = 1pt] {$1$};
    \node at (1.5,13) [font=\tiny,green!70!black,above = 1pt, inner sep = 1pt] {$2$};
    \node at (2.5,13) [font=\tiny,green!70!black,above = 1pt, inner sep = 1pt] {$3$};
    \node at (3.5,12) [font=\tiny,green!70!black,above = 1pt, inner sep = 1pt] {$5$};
    \node at (4.5,7) [font=\tiny,green!70!black,below = 1pt, inner sep = 1pt] {$11$};
    \node at (5.5,5) [font=\tiny,green!70!black,below = 1pt, inner sep = 1pt] {$14$};
    \node at (3,12.5) [font=\tiny,green!70!black,left = 1pt, inner sep = 1pt] {$4$};
    \node at (4,11.5) [font=\tiny,green!70!black,right = 1pt, inner sep = 1pt] {$6$};
    \node at (4,10.5) [font=\tiny,green!70!black,right = 1pt, inner sep = 1pt] {$7$};
    \node at (4,9.5) [font=\tiny,green!70!black,right = 1pt, inner sep = 1pt] {$8$};
    \node at (4,8.5) [font=\tiny,green!70!black,left = 1pt, inner sep = 1pt] {$9$};
    \node at (4,7.5) [font=\tiny,green!70!black,left = 1pt, inner sep = 1pt] {$10$};
    \node at (5,6.5) [font=\tiny,green!70!black,right = 1pt, inner sep = 1pt] {$12$};
    \node at (5,5.5) [font=\tiny,green!70!black,left = 1pt, inner sep = 1pt] {$13$};

\end{tikzpicture}%
}\medskip\\
\subcaptionbox{Step 3: The dominating condition corresponds to the signed area between the pair of paths never taking negative values. \label{subfig:step3}}[0.4\textwidth]{%
  \tikzset{vertex/.style={draw,shape=circle,fill=black,minimum size=2pt,inner sep=0}}
\begin{tikzpicture}[scale=0.37]
    \draw (0,13)--(0,0)--(14,0); 
    \foreach \x in {1,...,14}{\draw[very thin, gray] (\x,0)--(\x,13);} 
    \foreach \y in {1,...,13}{\draw[very thin, gray] (0,\y)--(14,\y);}

    \draw[dashed](0,0)--(13, 13); 

    \node at (.5,-.5){$\scriptstyle d_1$};\node at (1.5,-.5){$\scriptstyle d_2$};\node at (3,-.5){$\cdots$};

    \draw[red,thick] (0,13)--(1,13)--(1,11)--(2,11)--(2,10)--(4,10)--(4,9)--(5,9)--(5,7)--(6,7)--(6,6); 
    \draw[red,thick,dotted] (6,6)--(6,5);
    \draw[green!70!black,thick] (0,13)--(3,13)--(3,12)--(4,12)--(4,7)--(5,7)--(5,6)--(5,5)--(6,5); 

    \draw (6,5)--(8,5)--(8,4)--(13,4)--(13,3)--(14,3)--(14,0); 

    \node[vertex] at (0,13){}; 
    \node[left] at (0, 13) {$\scriptstyle n-1$};
    \node[vertex] at (14,0){};
    \node[below] at (14, 0) {$\scriptstyle n$};

    \node[vertex] at (6,6){};

    \draw[<->] (3.5,10.1)--(3.5,11.9);
    \node at (3.5,11.5) [left = 2pt, inner sep = 0pt, fill = white] {$\scriptstyle s_4\text{-}s'_4$}; 
\end{tikzpicture}%
}\hfill%
\subcaptionbox{Step 4: The cumulative distance process between the two paths gives the area.\label{subfig:step4}}[0.4\textwidth]{%
  \tikzset{vertex/.style={draw,shape=circle,fill=black,minimum size=2pt,inner sep=0}}
\begin{tikzpicture}[scale=0.37]
  \draw (0,13)--(0,0)--(14,0); 
  \foreach \x in {1,...,14}{\draw[very thin, gray] (\x,0)--(\x,13);} 
  \foreach \y in {1,...,13}{\draw[very thin, gray] (0,\y)--(14,\y);}

  \draw[dashed](0,0)--(13, 13); 

  \draw[thick,gray] (1,12)--(2,13) (1,11)--(3,13) (2,11)--(3,12) (2,10)--(4,12) (3,10)--(4,11)
  (4,8)--(5,9) (4,7)--(5,8) (5,6)--(6,7) (5,5)--(6,6);
  \node at (5,12) [above, fill = white, inner sep = 1pt] {$\scriptstyle{Y_5=2}$};

  \node at (.5,-.5){$\scriptstyle d_1$};\node at (1.5,-.5){$\scriptstyle d_2$};\node at (3,-.5){$\cdots$};

  \draw[red,thick] (0,13)--(1,13)--(1,11)--(2,11)--(2,10)--(4,10)--(4,9)--(5,9)--(5,7)--(6,7)--(6,6); 
  \draw[red,thick,dotted] (6,6)--(6,5);
  \draw[green!70!black,thick] (0,13)--(3,13)--(3,12)--(4,12)--(4,7)--(5,7)--(5,6)--(5,5)--(6,5); 

  \draw (6,5)--(8,5)--(8,4)--(13,4)--(13,3)--(14,3)--(14,0); 

  \node[vertex] at (0,13){}; 
  \node[left] at (0, 13) {$\scriptstyle n-1$};
  \node[vertex] at (14,0){};
  \node[below] at (14, 0) {$\scriptstyle n$};

  \node at (2,10.5) [font=\tiny,red,left = 1pt, inner sep = 1pt] {$5$};
  \node at (3.5,12) [font=\tiny,green!70!black,above = 1pt, inner sep = 1pt] {$5$};

  \node[vertex] at (6,6){};

\end{tikzpicture}%
}%
\caption{By using a number of reformulations, we show that the probability that a uniformly random sequence
$n-1\ge d_1\ge d_2\ge\dotsb\ge d_n\ge 0$ satisfies the dominating condition is equal to the probability that
the integral of a lazy simple symmetric random walk with $n-1$ steps, conditioned to end in $0$ or~$-1$,
does not take negative values.}
\label{fig:reformulation}
\end{figure}

We will now prove Lemma~\ref{lem:to_walk}. In order to get from a uniformly random non-increasing
sequence to a lazy random walk conditioned to end in $\{0,-1\}$, we need to make a few reformulations. 

\paragraph{Step 1: View a sequence as a lattice path.}
We start by viewing the non-increasing sequence as a lattice path $P$ from $(0,n-1)$ to $(n,0)$
which only takes steps to the right and downwards (see Figure~\ref{subfig:step1}).
Informally, we put $n$ stacks of heights $d_1,d_2,\dots,d_n$ respectively next to each other
and let the $P$ be the path from $(0,n-1)$ to $(n,0)$ that traces the outline of the stacks.
To be precise, the path begins with $n-1-d_1$ steps downwards before taking a step right.
For each $2\le i\le n$, the walk takes $d_{i-1}-d_i$ steps downwards on the line $x=i-1$
before taking a step right, and it ends with $d_n$ steps downwards to end at the point $(n,0)$.
There must be a unique $\ell$ such that the walk goes through the point $(\ell,\ell)$,
and it is not hard to see that $\ell$ is the largest $j$ such that $d_j\ge j$.

\paragraph{Step 2: Cut the path in half and reflect the latter half.}
Starting from $(0,n-1)$, the path takes $n-1$ steps to reach the point $(\ell,\ell)$
as each step decreases $y-x$ by~1, and we define a path $W$ which starts with these $n-1$
steps before ending with one final step down. We now define another path $W'$ starting
at $(0,n-1)$ using the other end of the path $P$. Starting at $(n,0)$, walk backwards along
the path $P$ to $(\ell,\ell)$ and, at each step, add the ``reflected" step to $W'$:
if the step on $P$ is vertical, add a right step to~$W'$; otherwise add a down step to~$W'$.
This gives two paths $W$ and $W'$ which take $n$ steps and both end at $(\ell,\ell-1)$
(see Figure~\ref{subfig:step2}, in which the steps on the two paths have been numbered).

\paragraph{Step 3: The dominating condition corresponds to the signed area between the pair of paths never taking negative values.}
If the walk $W$ (resp.\ $W'$) has a horizontal line from $(i-1,a)$ to $(i,a)$,
then $s_i = n-1-a$ (resp.\ $s_i'=n-1-a$). In particular, if $W$ is at $(i,a)$
after $k$ steps, the sum $\sum_{j=1}^is_j$ is exactly the area enclosed by the walk~$W$,
the line $y=n-1$ and the line $x=i$. Similarly, if $W'$ is at $(i,a)$
after $k$ steps, the sum $\sum_{j=1}^is'_j$ is exactly the area enclosed by the walk~$W'$,
the line $y=n-1$ and the line $x=i$. Therefore, if $W$ and $W'$ are both at $(i,a)$ after
$k$ steps, then the signed area between $W'$ and $W$ is exactly the
sum $\sum_{j=1}^i(s_j-s_j')$ (see Figure~\ref{subfig:step3}).
The dominating condition checks if this sum is non-negative for all~$i$. 

\paragraph{Step 4: The cumulative distance process between the two paths gives the area.}
Define $Z_1,\dots,Z_{n}$ by setting $Z_i$ equal to $+1$ if the walk $W$ goes down at
the $i$th step, and $-1$ if it goes right. Similarly, define $Z_1',\dots,Z_{n}'$ by
setting $Z_i'$ equal to $-1$ if the walk $W'$ goes down at the $i$th step, and $+1$
otherwise. Then,
$\lw_i=\tfrac{1}{2}\sum_{j=1}^i (Z_j + Z_j')$ keeps track of the (signed) number of diagonal
right/up steps from the walk $W$ to the walk $W'$. When the two walks coincide at time~$k$,
there will be a diagonal line through every box between the two walks, and hence the number
of diagonal lines is equal to the signed area between $W'$ and~$W$ up to time $k$
(see Figure~\ref{subfig:step4}). We claim that we only need to check the dominating
condition when the two walks coincide. Indeed, between times at which the walks coincide,
the area process is monotone so if the area process first takes a negative value at some time $i$ where the walks do not coincide, the process will still be negative when the walks next coincide.
However, the number of diagonal lines used up to time $k$ is the integral of $\lw$ up to
time~$k$, and it suffices to check that condition $\sum_{i=1}^k \lw_i \ge 0$ at all times $k$ where the walks coincide. Since the integral of $Y$ is also monotone on excursions away from zero (i.e. between times when $W$ and $W'$ coincide), the condition (\ref{eq:DC}) is equivalent to $\sum_{i=1}^k \lw_i \ge 0$ for all $k \le n$. But $Y_n=0$, so it is in fact sufficient that 
\begin{equation}\label{eq:WDC}
 \sum_{i=1}^k \lw_i \ge 0 \quad \text{for all }k \le n-1.
\end{equation}

\paragraph{Step 5: The parity condition corresponds to the integral of $Y$ being even}
We now claim that
\begin{equation}\label{eqn:parity}
    \sum_{i=1}^n d_i\equiv \sum_{i=1}^{n-1} \lw_i \mod 2.
\end{equation}
Split the sum $\sum_{i=1}^n d_i$ into $A_{\text{red}}=\sum_{i=1}^{\ell}d_i$ and $A_{\text{green}}=\sum_{i=\ell+1}^{n}d_i$. In Figure \ref{subfig:step1}, $A_\text{red}$ is the area below the red curve (left of $(\ell,\ell)$) and $A_{\text{green}}$ is the area below the green curve (right of $(\ell,\ell)$). Clearly, $A_\text{red}$ is also the area under the walk $W$. The area under the walk $W'$ is given by $A_\text{green} + \ell(\ell - 1)$. Hence, the signed area enclosed by $W'$ and $W$ is given by \[(A_\text{green}+\ell(\ell-1))-A_\text{red}\equiv A_\text{green}+A_\text{red} \mod 2.\]
It was shown in the previous step that the signed area enclosed by $W'$ and $W$ is given by $\sum_{i=1}^{n} \lw_i=\sum_{i=1}^{n-1} \lw_i$, and the claim follows.

\paragraph{Step 6: The distance process between a random pair of paths is a conditioned lazy random walk.}
Finally, we need to understand the distribution of $\lw$ when we sample a sequence
$n-1\ge d_1\ge\dotsb\ge d_n\ge 0$ uniformly at random. First, observe that Step~2 gives a bijection
between such non-increasing sequences and pairs of lattice paths $(W,W')$ of $n$ steps that
start at $(0,n-1)$ and end at the same point, and such that $W$ ends by taking a downwards step. These are in turn in bijection with pairs of lattice paths $(W, W')$ with $n - 1$ steps that start at $(0, n-1)$ and for which the corresponding walk $(Y_k)_{k=0}^{n-1}$ has $Y_{n-1} \in \{0, -1\}$. 
Therefore, sampling a uniformly random sequence
corresponds to sampling a uniformly random pair of such paths. If we ignore the
requirement that $Y_{n-1} \in \{0, -1\}$, then for any~$j$, $Z_j$ and
$Z'_j$ have opposite signs with probability $1/2$, and both have value $+1$ (or $-1$)
with probability $1/4$, making $\lw$ a lazy simple symmetric random walk.
To recover the actual distribution of (the first $n-1$ steps of) $\lw$, we need to restrict to the paths for which $Y_{n-1} \in \{0, 1\}$, and we simply condition on this being the case.

\bigskip

We remark that steps 1 through 5 give a deterministic mapping which maps a sequence $n -1 \geq d_1 \geq \dotsb d_n \geq 0$ to a walk $(Y_i)_{i=1}^n$ and we have shown that we can check if the sequence is graphic by checking certain properties of the walk. 
In Section \ref{sec:comp}, we use this reformulation (without introducing randomness) to enumerate the number of graphic sequences for small $n$. 

\section{Proof of main result}
\label{sec:proof}
In this section, we prove Theorem~\ref{thm:main}, which is a direct consequence of the following two results.  

\begin{prop}\label{prop:areanonnegrectangle}
 A uniformly random sequence $n-1\ge d_1\ge\dotsb\ge d_n\ge 0$ has probability
 \[
  (1+o(1))\frac{\Gamma(3/4)}{\sqrt{2\pi(1-\rho)}}n^{-1/4}
 \]
 of satisfying the dominating condition \eqref{eq:DC}.
\end{prop}

\begin{lemma}\label{lem:parity}
 A uniformly random sequence $n-1\ge d_1\ge\dotsb\ge d_n\ge 0$ which satisfies the
 dominating condition \eqref{eq:DC}, has probability $1/2+o(1)$ of being a
 graphic sequence \textup(equivalently, of having $\sum_{i=1}^nd_i$ even\textup).
\end{lemma}

The remainder of this section is structured as follows. In Section~\ref{subsec:useful}
we introduce a lemma (due to Burns \cite{Burns2007TheNO}) on exchangeable random sequences that turns out to be very
useful in our proofs.
We use the lemma in Section~\ref{subsec:firstasymptotics} to show that $G(n)=\Theta(4^nn^{-3/4})$. Finally, in Section~\ref{subsec:proofoverview}, we use
the lemma again to prove Proposition \ref{prop:areanonneg}, which states that, for $\lwa$ the area process of a lazy simple symmetric random walk $\lw$, as $n\to \infty$,
 \begin{equation}\label{eq:convprob}
  n^{1/4}\Prb(\lwa_1,\dots,\lwa_{n}\ge 0\mid \lw_n=0) \to \frac{\Gamma(3/4)}{\sqrt{2\pi(1-\rho)}}.
\end{equation}
This equation will be used to prove Proposition \ref{prop:areanonnegrectangle} in Section \ref{subsec:squaretorectangle}.
We also postpone the proofs of Lemma~\ref{lem:parity} and the lemmas used in the proof of Proposition~\ref{prop:areanonneg} to Section~\ref{sec:technical}. An overview of the notation is given in Appendix \ref{sec:notation_overview}. 

\subsection{A useful lemma}\label{subsec:useful}
We will make use of the following useful lemma that appears in Sections 2.3--2.4
of~\cite{Burns2007TheNO}. We include a proof in the appendix for completeness.

\begin{restatable}{lemma}{useful}\label{lem:useful}
 Let $x=(x_1,\dots,x_n)\in\R^n$, let $\sigma$ be a uniformly random permutation of\/ $[n]$
 and let $s=(s_1,\dots,s_n)$ be an independent uniformly random element of $\{-1,1\}^n$. Then 
\[
       \Prb\left(\sum_{i=1}^k s_i x_{\sigma(i)}\ge 0\text{ for all }k\in [n]\right)\ge \frac{(2n-1)!!}{2^n n!}
\] and
 \[
  \Prb\left(\sum_{i=1}^k s_i x_{\sigma(i)}> 0\text{ for all }k\in [n]\right)\leq \frac{(2n-1)!!}{2^n n!}.
 \]
 If, additionally, for all distinct $A,A'\subseteq [n]$, the
 corresponding sums are also distinct, i.e.\ $\sum_{i\in A}x_i\ne \sum_{i\in A'}x_i$, then 
\[
       \Prb\left(\sum_{i=1}^k s_i x_{\sigma(i)}\ge 0\text{ for all }k\in [n]\right)= \frac{(2n-1)!!}{2^n n!}
\]
and 
\[ 
 \Prb\left(\sum_{i=1}^k s_i x_{\sigma(i)}> 0\text{ for all }k\in [n]\right)= \frac{(2n-1)!!}{2^n n!}.
\]
 
\end{restatable}

We will apply this lemma to sequences of exchangeable random variables of which the law
is invariant under sign changes of the elements, so we use the following equivalent form of the lemma.

\begin{lemma}\label{lem:useful2}
 Let $(X_1,\dots,X_n)$ be a random variable in $\R^n$ such that for any $\sigma\in S_n$
 and any $s\in\{-1,1\}^n$, 
 \[
  (X_1,\dots,X_n)\overset{d}{=}(s_1X_{\sigma(1)},\dots,s_nX_{\sigma(n)}).
 \]
 Then, 
 \begin{equation}
 \label{eq:useful_lemma_1}
  \Prb\left(\sum_{i=1}^k X_i\ge 0\text{ for all\/ }k\in [n]\right)\ge \frac{(2n-1)!!}{2^n n!}
 \end{equation}
 and 
  \begin{equation}\label{eq:useful_lemma_2}
  \Prb\left(\sum_{i=1}^k X_i > 0\text{ for all\/ }k\in [n]\right) \leq  \frac{(2n-1)!!}{2^n n!}.
 \end{equation}
  If, additionally, almost surely, for all distinct $A,A'\subseteq [n]$, the
 corresponding sums are also distinct, i.e.\ $\sum_{i\in A}X_i\ne \sum_{i\in A'}X_i$, then 
\begin{equation}
\label{eq:useful_lemma_3}
  \Prb\left(\sum_{i=1}^k X_i\ge 0\text{ for all\/ }k\in [n]\right)= \frac{(2n-1)!!}{2^n n!}
 \end{equation}
 and 
  \begin{equation}\label{eq:useful_lemma_4}
  \Prb\left(\sum_{i=1}^k X_i > 0\text{ for all\/ }k\in [n]\right) =  \frac{(2n-1)!!}{2^n n!}.
 \end{equation}
 
\end{lemma}
In the following sections, we will apply this lemma to the sequence of excursion areas of a lazy simple symmetric random walk bridge. To be precise, let $\lb = (\lb_k)_{k=1}^{n}$ be the first $n$ steps of the lazy SSRW $ \lw = (\lw_k)_{k \geq 1}$ conditioned on the event $\lw_n=0$,  so that $\lb$ is a lazy simple symmetric random walk bridge of length $n$. Let $N_n$ be the number of times that the  bridge $\lb$ hits~$0$ (after time 0).
Condition on the event that $N_n = N$. Then the walk $\lb$ has $N$ excursions with areas
$\lbae_1,\lbae_2,\dots,\lbae_N$ say. Since the walk $\lba_k = \sum_{i=1}^k \lb_i$ is monotone during the individual excursions, the walk $\lbac=(\lbac_k)_{k=1}^N$, defined by \[
\lbac_k=\sum_{i=1}^k\lbae_i,
\]
is never negative if and only if the walk $\lba$ is never negative. The increments of $\lbac$ $(\lbae_1,\lbae_2,\dots,\lbae_N)$ are exchangeable and their law is invariant under sign changes, so we can use Lemma \ref{lem:useful2} to estimate 
$\Prb\left(\lbac_1,\dots, \lbac_N\geq 0\right)$,  which is equal to the probability we want to calculate: $\Prb(\lba_1,\dots,\lba_n\geq 0)$.

To see that Lemma \ref{lem:useful2} indeed implies Lemma \ref{lem:useful}, observe that for $x$, $s$ and $\sigma$ as in the statement of Lemma \ref{lem:useful}, the random variable $(X_1,\dots,X_n):=(s_1x_{\sigma(1)},\dots,s_nx_{\sigma(n)})$ satisfies the conditions of Lemma \ref{lem:useful2}. We now show that Lemma~\ref{lem:useful} implies Lemma~\ref{lem:useful2}.

\begin{proof}[Proof of Lemma \ref{lem:useful2}]
We observe that, for $\sigma$ a uniformly random permutation of $[n]$ and $s=(s_1,\dots,s_n)$
an independent uniformly random element of $\{-1,1\}^n$,
\begin{align*}
 \Prb\left(\sum_{i=1}^k X_i\ge 0\text{ for all }k\in [n]\right)
 &=\Prb\left(\sum_{i=1}^k s_i X_{\sigma(i)}\ge 0\text{ for all }k\in [n]\right)\\
 &=\E\left[\Prb\left(\sum_{i=1}^k s_i X_{\sigma(i)}\ge 0\text{ for all }k\in [n]\mid X_1,\dots,X_n\right)\right],
\end{align*}
so that \eqref{eq:useful_lemma_1} and \eqref{eq:useful_lemma_3} follow by applying Lemma~\ref{lem:useful} to
\[
 \Prb\left(\sum_{i=1}^k s_i X_{\sigma(i)}\ge 0\text{ for all }k\in [n]\mid X_1,\dots,X_n\right).
\]
The derivation of \eqref{eq:useful_lemma_2} and \eqref{eq:useful_lemma_4} is similar.
\end{proof}

We make the following observation about the value of the bounds in Lemmas~\ref{lem:useful} and~\ref{lem:useful2}. Note that by taking either an even or odd number of terms in Wallis's product formula
\[
 \frac{\pi}{2}=\frac{2}{1}\cdot \frac{2}{3}\cdot \frac{4}{3}\cdot \frac{4}{5}\cdot \frac{6}{5}\cdots
\]
one obtains alternately lower and upper bounds for $\frac{\pi}{2}$. Rearranging these inequalities 
gives that for all $n\ge1$
\begin{equation}\label{e:ulbounds}
 \frac{1}{\sqrt{\pi(n+1/2)}}\le \frac{(2n-1)!!}{2^nn!}=\binom{2n}{n}\frac{1}{4^n} \le \frac{1}{\sqrt{\pi n}}.
\end{equation}

\subsection{First asymptotics}\label{subsec:firstasymptotics}
In this section, we determine the asymptotics of the number of sequences that satisfy
the dominating condition, as well as $G(n)$, up to a constant factor.

\begin{prop}\label{prop:firstasymptotics}
The number of graphic sequences of length $n$ is $G(n)=\Theta(4^n/n^{3/4})$.
\end{prop}

The main ingredient in the proof of the above proposition is the following, simpler statement.
\begin{prop}\label{prop:basic}
 For $(\lw_i)_{i=1}^n$ the lazy SSRW and $\lwa_k=\sum_{i=1}^k\lw_i$ its area process,
 \[
  \Prb\big(\lwa_1,\dots,\lwa_n\ge 0 \mid \lw_{n} =0\big)=\Theta(n^{-1/4}).
 \]
\end{prop}
We will prove this proposition using \eqref{eq:useful_lemma_1} and \eqref{eq:useful_lemma_2} from Lemma \ref{lem:useful2}. (It can also be deduced from the more general result \cite[Proposition 1]{Vysotsky2014}.)  Before we get to that, we first use it to prove Proposition~\ref{prop:firstasymptotics}.
We cannot immediately apply it to calculate the order of growth of the number of graphic sequences as we need to condition on the walk ending at either $0$ or~$-1$, not only $0$. The following lemma shows that this change in conditioning only changes the probability by a constant factor.

\begin{lemma}\label{lem:ignore1} 
 For all $n \ge 1$, 
 \[
  \tfrac{1}{2} \le
  \frac{\Prb(\lwa_1,\dots,\lwa_{n}\ge 0 \mid \lw_{n}\in\{0,-1\})}{\Prb(\lwa_1,\dots,\lwa_{n}\ge 0 \mid \lw_{n}=0)}
  \le 1.
 \]
\end{lemma}

The proof of Lemma \ref{lem:ignore1} is elementary, but for the sake of brevity we postpone it to Section~\ref{sec:technical} and skip straight to the proof of Proposition~\ref{prop:firstasymptotics}.

\begin{proof}[Proof of Proposition \ref{prop:firstasymptotics}]
By Lemma~\ref{lem:to_walk}, the probability that a uniformly random sequence
$n-1\ge d_1\ge\dotsb\ge d_n\ge 0$ satisfies the dominating condition \eqref{eq:DC} is equal to
$\Prb(\lwa_1,\dots,\lwa_{n-1}\ge 0 \mid \lw_{n-1}\in\{0,-1\})$, which is $\Theta(n^{-1/4})$ by Proposition~\ref{prop:basic} and
Lemma~\ref{lem:ignore1}. Hence, the number of sequences which satisfy the dominating condition is $\Theta(4^n/n^{3/4})$.

Clearly, $G(n) = O(4^n/n^{3/4})$ is immediate, and we only need to show a corresponding lower bound.
Let $H(n)$ denote the number of sequences
$(d_i)_{i=1}^n$ that satisfy the dominating condition for which $\sum_{i=1}^nd_i$ is odd. The number of
sequences of length $n$ that satisfy the dominating condition is therefore $G(n)+H(n)=\Theta(4^n/n^{3/4})$. 
Each sequence $(d_i)_{i=1}^n$ in $G(n)$ gives rise to a unique sequence in $G(n+1)$ by appending $d_{n+1} = 0$.
Separately, each sequence $(d_i)_{i=1}^n$ in $H(n)$ gives rise to a unique sequence in $G(n+1)$ by adding an extra 1 to the graphic sequence (in the appropriate place). Hence, 
\[
 G(n+1)\ge \max\{G(n),H(n)\}\ge \tfrac12(G(n)+H(n)).
\]
Hence, $G(n) \geq \frac{1}{2} (G(n - 1)+H(n - 1))=\Theta(4^n/n^{3/4})$ as required.
\end{proof}

\begin{proof}[Proof of Proposition~\ref{prop:basic}] 
Recall from Section~\ref{subsec:useful} that $\lbac_1,\dots, \lbac_{N_n}$ is the process with the excursion areas of a lazy simple symmetric random walk bridge of length $n$ as its increments and that 
\[\Prb\left(\lwa_1,\dots, \lwa_n\geq 0\mid \lw_n=0 \right)=\Prb\left(\lbac_1,\dots, \lbac_{N_n}\geq 0\right).\]

Recall that the increments of $\lbac$ are exchangeable and their law is invariant under sign changes so, conditional on the event that $N_n = N$, the probability that $\lbac$ never takes a negative value is at least \[\frac{(2N-1)!!}{2^{N}N!}\]
by \eqref{eq:useful_lemma_1} of Lemma \ref{lem:useful2}. This implies that 
\[\Prb(A_1, \dots, A_n \geq 0 \mid  Y_n = 0) \geq \E\left[\frac{(2N_n-1)!!}{2^{N_n}N_n!}\right].
\]
A little care is needed in computing the expectation and we postpone the proof that 
$\E\left[\frac{(2N_n-1)!!}{2^{N_n}N_n!}\right]
=(1+o(1)) n^{-1/4}\frac{\Gamma(3/4)}{\sqrt{2\pi}}$
to Lemma~\ref{lem:probnonnegperturbed}.

For the upper bound we use \eqref{eq:useful_lemma_2} from Lemma~\ref{lem:useful2}, but we need to use a trick to circumvent the strict inequality in 
the event. Consider the walks of length $n+2$ that begin with an upwards step and then a downwards step.
After these first two steps the walk behaves like a lazy SSRW of length $n$ conditioned to end at~0,
but the area process is one higher. This means the event $\{\sum_{i=1}^k\lw_i>0,\,k=1,\dots,n+2\}$ is exactly
the event $\{\sum_{i=3}^k\lw_i\ge 0,\,k=3,\dots,n+2\}$. A simple calculation shows that the probability
the walk goes up and then down given that $\lw_{n+2} = 0$ is at least $1/16$ for all $n$. It follows that
\[
 \Prb(\lwa_1,\dots,\lwa_{n}\ge 0 \mid \lw_{n}=0)\leq 16 \Prb(\lwa_1,\dots,\lwa_{n+2}>0 \mid \lw_{n+2}=0).
\]
Therefore, if $N_{n+2}$ is the number of times a lazy SSRW of length $n+2$ hits zero when conditioned to end at zero, similar reasoning to before shows that \eqref{eq:useful_lemma_2} from Lemma~\ref{lem:useful2} and Lemma~\ref{lem:probnonnegperturbed} imply that 
 \[\Prb(\lwa_1,\dots,\lwa_{n}\ge 0 \mid \lw_{n}=0)\leq 16 \E\left[\frac{(2N_{n+2}-1)!!}{2^{N_{n+2}} N_{n+2}!}\right] =(1+o(1)) n^{-1/4} \frac{16 \Gamma(3/4)}{\sqrt{2\pi}}.\]
The claimed statement follows. 
\end{proof}

\subsection{Proof of Proposition \ref{prop:areanonneg}}\label{subsec:proofoverview}
 We will use Lemma~\ref{lem:useful2} to get tighter estimates for the probability that $\lbac$
never takes a negative value. The weak inequality (\ref{eq:useful_lemma_1}) from Lemma~\ref{lem:useful2} is only strengthened to the equality \eqref{eq:useful_lemma_3} when all subset sums of elements
in $(\lbae_i:1\le i \le N_n)$ are distinct almost surely. However, this is not the case: for $n\ge 2$, with positive probability, $\lb_0=\lb_1=\lb_2=0$, in which case both $\lbae_1$ and $\lbae_2$ equal $0$. We overcome this issue by
perturbing each $\lbae_i$ by a small random amount.

 To be precise, let $\eps_1,\eps_2,\dots$ be i.i.d.\ 
$\Unif[-\tfrac{1}{2n},\tfrac{1}{2n}]$ random variables and define $\pbae_i=\lbae_i+\eps_i$ and
$\pbac_k=\sum_{i=1}^k \pbae_i$. Then, as the perturbations are independent and  $\eps_i\overset{d}{=} -\eps_i$, the increments of $\pbac$ are exchangeable and their law is invariant under sign changes.  Moreover, since $\Unif[-\tfrac{1}{2n},\tfrac{1}{2n}]$ is non-atomic, for any sequences of real numbers $x_1,\dots,x_N$, the probability that  there are distinct subsets $S,T\subseteq \{1,\dots,N\}$ with $\sum_{i\in S}(x_i+\eps_i) = \sum_{j\in T}(x_j+\eps_j)$ is $0$. This means that conditional on the event $N_n = N$, the equality \eqref{eq:useful_lemma_3} from 
Lemma~\ref{lem:useful2} shows that the probability that  $\pbac_1,\dots,\pbac_N\geq 0$ 
 is exactly 
$
 \frac{(2N-1)!!}{2^{N}N!}.
$
This implies that 
\[
 \Prb\big(\pbac_1,\dots,\pbac_{N_n}\ge 0\big)
 =\E\left[\frac{(2N_n-1)!!}{2^{N_n}N_n!}\right].
\]

We compute the expectation in Lemma~\ref{lem:probnonnegperturbed}, where we show that as $n\to\infty$, 
\[
\E\left[\frac{(2N_n-1)!!}{2^{N_n}N_n!}\right]=(1+o(1))n^{-1/4} \frac{\Gamma(3/4)}{\sqrt{2\pi}}.
\]

Combining the two equations above gives
\begin{equation}\label{eq:asym_perturbed}
 \Prb\big(\pbac_1,\dots,\pbac_{N_n}\ge 0\big) =(1+o(1))n^{-1/4} \frac{\Gamma(3/4)}{\sqrt{2\pi}}.
\end{equation}
Note that by our choice for the law of the perturbations $\{\varepsilon_i\}$ we have $|\sum_{i=1}^k \varepsilon_i|<1$ for all $k$.
Since  $\pbac_k=\lbac_k+\sum_{i=1}^k \varepsilon_i$ and $\lbac$ only takes integer values, this implies that $\lbac_k \geq 0$ whenever $\lbac_k \geq 0$.
Therefore, to compute the probability that $\lbac$ never takes a negative value we can use the following equality
\begin{align*}
 \Prb\big(\pbac_1,\dots,\pbac_{N_n}\ge 0\big)
 &=\Prb\big(\pbac_1,\dots,\pbac_{N_n},\lbac_1,\dots,\lbac_{N_n}\ge 0 \big)
 \\&=\Prb\big(\lbac_1,\dots,\lbac_{N_n}\ge 0\big)
 \Prb\big(\pbac_1,\dots,\pbac_{N_n}\ge 0 \mid \lbac_1,\dots,\lbac_{N_n}\ge 0\big).
\end{align*}
Equation \eqref{eq:asym_perturbed} gives the asymptotic value for $\Prb\big(\pbac_1,\dots,\pbac_{N_n}\ge 0\big)$, so our result will follow by evaluating the limit of
$\Prb(\pbac_1,\dots,\pbac_{N_n}\ge 0\mid\lbac_1,\dots,\lbac_{N_n}\ge 0)$ as $n\to \infty$.
For this, we need a further observation. Again using that $|\sum_{i=1}^k \varepsilon_i|<1$ for all $k$ and that $\lbac$ only takes integer values, $\lbac_k>0$ implies that $\pbac_k>0$. Therefore, for any $k$,
\[\{\pbac_k\ge 0\}=\{\lbac_k>0\}\cup \left\{\lbac_k=0, \sum_{i=1}^k \varepsilon_i\ge 0\right\}.\]
This implies that on the event that $\lbac_1,\dots,\lbac_{N_n}\ge 0$, the event  $\lbac_1,\dots,\lbac_{N_n}\ge 0$ holds if and only if $\sum_{i=1}^k \varepsilon_i\ge0$ for all $k$ for which $\lbac_k=0$. 
Suppose that $\lbac$ is equal to zero $M_n$ times, namely at at $\xi_1,\dots,\xi_{M_n}$, and let us also set $\xi_0 = 0$.
By definition, $\pbac_{\xi_k} = \sum_{i=1}^{\xi_k}\eps_i$, and the  increment between
times $\xi_{k-1}$ and $\xi_k$ is exactly $\eta_k=\sum_{i=\xi_{k-1}+1}^{\xi_k}\eps_i$.

We show that we can apply the equality \eqref{eq:useful_lemma_3} from Lemma~\ref{lem:useful2} to the process $(\pbac_{\xi_k},1\le k \le M_n)$. Firstly, the sequence $(\eta_k)_{1\le k\le M_n}$ is exchangeable as $(\xi_i-\xi_{i-1},1\le i \le M_n)$ is exchangeable and the perturbations $\eps_i$ are i.i.d. Secondly, since $\varepsilon_i\overset{d}{=}-\varepsilon_i$, the law of $(\eta_k)_{1\le k\le M_n}$ is invariant under sign changes of the elements. Finally, we already observed that all subset sums of $\{\varepsilon_1,\dots,\varepsilon_{N_n}\}$ are distinct almost surely.  This means that we can indeed apply the equality \eqref{eq:useful_lemma_3} to  $(\pbac_{\xi_k},1\le k \le M_n)$ conditional on $M_n=M$, that is,  
\[
\Prb(\pbac_{\xi_1},\dots,\pbac_{\xi_{M_n}}\ge 0\mid M_n=M)=\frac{(2M-1)!!}{2^{M}M!}.
\]
From the definition of $(\xi_i)$ as the times $i$ where $\lbac_{i}$ and $\pbac_i$ may differ in signs, it follows that
\[
 \Prb\big(\pbac_1,\dots,\pbac_{N_n}\ge 0\mid \lbac_1,\dots,\lbac_{N_n}\ge 0,\,M_n=M\big)
 =\frac{(2M-1)!!}{2^{M}M!}.
\]
This implies  
\[\Prb\big(\pbac_1,\dots,\pbac_{N_n}\ge 0\mid \lbac_1,\dots,\lbac_{N_n}\ge 0\big)
 =\E\left[\left.\frac{(2M_n-1)!!}{2^{M_n}M_n!}\,\right|\, \lbac_1,\dots,\lbac_{N_n}\ge 0 \right ].\]
What remains is to understand the distribution of $M_n$ conditional on the event that $\{\lbac_1,\dots,\lbac_{N_n}\ge 0\}$. 
In fact, we will show that it converges in distribution to $G\sim\Geom(\rho)$. 
For this it is enough to show that for $\ell\ge 1$,
\begin{equation}\label{eq:ellconvtorho}
 \Prb(M_n\ge \ell+1 \mid \lbac_1,\dots,\lbac_{N_n}\ge 0,\,M_n\ge \ell)\to \rho.
\end{equation}
 Lemma~\ref{lem:pn_to_rho} states that
\[
 \probareazero{n}:= \Prb(M_n\ge 1\mid \lbac_1,\dots,\lbac_{N_n}\ge 0)
\]
converges to $\rho$ as $n\to \infty$. 
Suppose $M_n\ge\ell$ and let $\zeta_\ell$ be the time at which $\lb$ and $\lba$ hit $0$
simultaneously for the $\ell$th time. Then, $\lb$ and $\lba$ restricted to $[\zeta_\ell,n]$ are distributed as a
lazy simple symmetric random walk bridge with $n-\zeta_\ell$ steps and its area process respectively, so 
\[
 \Prb\big(M_n\ge\ell+1 \mid \lbac_1,\dots,\lbac_{N_n}\ge 0,\, M_n\ge\ell,\,\zeta_\ell)
 =\probareazero{n-\zeta_\ell}.
\]
Lemma~\ref{lem:hittingtimeat0} implies that (conditional on the event
$\{\lbac_1,\dots,\lbac_{N_n}\ge 0\}$) it holds that $\zeta_{M_n}=O(n^{-1/2})$ in probability, so $n-\zeta_\ell=(1+o(1)) n$ and \eqref{eq:ellconvtorho} follows from
$\probareazero{n}\to \rho$. This implies that 
\[
 \Prb\big(\pbac_1,\dots,\pbac_{N_n}\ge 0\mid \lbac_1,\dots,\lbac_{N_n}\ge 0\big)
 \to \E\left[\frac{(2G-1)!!}{2^G G!}\right],
\]
and a simple calculation shows that this equals $\sqrt{1-\rho}$. Putting this all together gives
\begin{align*}
 \Prb\big(\lbac_1,\dots,\lbac_{N_n}\ge 0\big)
 &=\frac{\Prb\big(\pbac_1,\dots,\pbac_{N_n}\ge 0\big)}{
  \Prb\big(\pbac_1,\dots,\pbac_{N_n}\ge 0\mid \lbac_1,\dots,\lbac_{N_n}\ge 0\big)}\\
 &=(1+o(1)) \frac{\Gamma(3/4)}{\sqrt{2\pi(1-\rho)}}n^{-1/4},
\end{align*}
as claimed.

\section{The postponed proofs}
\label{sec:technical}
 We will first introduce a coupling between simple symmetric random walks (resp.\ bridges)
and lazy simple symmetric random walks (resp.\ bridges) that turns out to be useful in our proofs.  
Let $(\uw_i)_{i\ge 0}$ be a simple symmetric random walk. Then, if we set $\lw_i=\frac{1}{2}\uw_{2i}$
for each $i$ we see that $(\lw_i)_{i\ge 0}$ has the law of a lazy simple symmetric random walk.
This implies that for $k\in \Z$
\begin{equation}\label{eq:prb_Y}
\Prb(\lw_n=k)=\Prb(\uw_{2n}=2k)=4^{-n}\binom{2n}{n+k}.
\end{equation}
Moreover, since $\uw$ is zero only at even times, the zeroes of $\uw$ are in one-to-one correspondence
with the zeroes of~$\lw$, and in particular, if $(\ub_i)_{0\le i\le 2n}$ is a simple symmetric
random walk bridge and $\lb_i=\frac{1}{2}\ub_{2i}$, we have that $(\lb_i)_{0\le i\le n}$ is a
lazy simple symmetric random walk bridge. 

To prove Lemma~\ref{lem:ignore1} we first show that the probability of event
$\{\lwa_1,\dots,\lwa_{n}\ge 0\}$ is monotone with respect to the value of $\lw_{n}$.
\begin{lemma}\label{lem:monotoneinY_n}
 For all $-n\le k \le k' \le n$,
 \[\textstyle
  \Prb(\lwa_1,\dots,\lwa_{n}\ge 0 \mid \lw_{n} = k)
  \le \Prb(\lwa_1,\dots,\lwa_{n}\ge 0 \mid \lw_{n} = k').
 \]
\end{lemma}
\begin{proof}
It suffices to prove the claim when $k' = k+1$. We will use the coupling between the SSRW and the lazy SSRW introduced at the beginning of this section. Consider the SSRW $\uw=(\uw_k)_{i=1}^{2n}$ conditioned on the event $\uw_{2n}=2k$, and define $\lw$ by $\lw_i:=\tfrac12\uw_{2i}$. Then $\lw$ is distributed as a lazy SSRW with $n$ steps conditioned on the event $\lw_{n}=k$. 

Now, observe that the increments of $\uw$ conditioned on $\uw_{2n}=2k$ are distributed as a uniform ordering of $n+k$ instances of $1$ and $n-k$ instances of $-1$. We consider a modified sequence obtained from the original sequence by picking a uniformly random $-1$ and changing it to a $+1$, so that the modified sequence has the law of the increments of a SSRW of $2n$ steps conditioned to end in $2k+2$, which, again via the coupling, corresponds to a lazy SSRW of $n$ steps conditioned to end in $k+1$. We see that our modification increased one increment of $\lw$ by $1$, and since the event $\{\lwa_1,\dots,\lwa_n\geq 0\}$ is increasing in the increments of $\lw$, the result follows. \end{proof}

\begin{proof}[Proof of Lemma~\ref{lem:ignore1}]
We abbreviate $\event=\{\lwa_1,\dots,\lwa_{n}\ge 0\}$, and start with the lower bound. We use \eqref{eq:prb_Y}.
\begin{align*}
 \Prb(\event \mid \lw_{n}\in\{0,-1\})
 &=\frac{\Prb(\event \mid \lw_{n} = 0)\Prb(\lw_{n} = 0)
 + \Prb(\event \mid \lw_{n} = -1)\Prb(\lw_{n}= -1)}{\Prb(\lw_{n} \in \{0,-1\})}\\
 &\ge \frac{\Prb(\event \mid \lw_{n} = 0)\binom{2n}{n}}{\binom{2n+1}{n}}\\
 &=\tfrac{n+1}{2n+1} \Prb(\event \mid \lw_{n} = 0)\\
 &\ge \tfrac{1}{2} \Prb(\event \mid \lw_{n} = 0).
\intertext{We now turn to the upper bound, for which we use Lemma~\ref{lem:monotoneinY_n}. }
 \Prb(\event \mid \lw_{n}\in\{0,-1\})
 &=\frac{\Prb(\event \mid \lw_{n} = 0)\Prb(\lw_{n} = 0)
 + \Prb(\event \mid \lw_{n} = -1)\Prb(\lw_{n}= -1)}{\Prb(\lw_{n} \in \{0,-1\})}\\
 &\le \frac{\Prb(\event \mid \lw_{n} = 0)(\Prb(\lw_{n} = 0) + \Prb(\lw_{n} = -1))}{\Prb(\lw_{n} \in \{0,-1\})}\\
 &=\Prb(\event \mid \lw_{n} = 0).\qedhere
\end{align*}
\end{proof}

\subsection{Lemmas for Proposition~\ref{prop:areanonneg}}

As before, let  $N_n=|\{1\leq i\leq n:\lb_i=0\}|$ be the number of returns to $0$ of $\lb$ up to time~$n$.
We have the following lemma. 

\begin{lemma}\label{lem:convtorayleigh}
 As $n\to\infty$, $n^{-1/2}N_n\overset{d}{\to} 2\sqrt{E}$ where $E\sim\Exp(1)$ has a standard
 exponential distribution. Moreover, for any $\gamma\ge0$, 
 \[
  \Prb\big(N_n<\gamma n^{1/2}\big)\le \tfrac{\gamma^2}{2}.
 \]
\end{lemma}
\begin{proof}
We fix an $n$ and a $k\le n$. We will count the number of Bernoulli bridges with $2n$ steps and at least $k$ returns to $0$ (i.e.\ at least $k$ excursions away from~$0$). For a Bernoulli bridge with $2n$ steps
and at least $k$ excursions, flip the last $k$ excursions so that they are positive. Then remove the
last step of each of the last $k$ excursions. Each of the removed steps was downward, so we now obtain
a path of length $2n-k$ that ends at level~$k$. We can recover the original position of the removed steps:
the $i$th removed step should be included after the last time the path is at level~$i$. Therefore,
each path of length $2n-k$ that ends at level $k$ corresponds to $2^k$ bridges with more than $k$ zeroes,
so the number of bridges with $2n$ steps and more than $k$ zeroes equals $2^k\binom{2n-k}{n}$. Thus
the probability that a simple symmetric random walk bridge with $2n$ steps returns to $0$ at least $k$ times equals 
\[
 \frac{2^k\binom{2n-k}{n}}{\binom{2n}{n}}
 =\frac{2^k n(n-1)\dotsm(n-k+1)}{(2n)(2n-1)\dotsm(2n-k+1)}
 =\prod_{i=1}^{k-1}\left(1-\frac{i}{2n-i}\right).
\]

By the coupling between a lazy random walk bridge on $[n]$ and a simple symmetric random walk bridge
on $[2n]$ that preserves the number of zeroes, we then see that for $k=O(\sqrt{n})$,
\[
 \log\Prb(N_n\geq k)=\log\prod_{i=1}^{k-1}\left(1-\frac{i}{2n-i}\right)
 =-\sum_{i=1}^{k-1}\Big(\frac{i}{2n}+O(i^2/n^2)\Big)=-\frac{k^2}{4n}+O(n^{-1/2}),
\]
so $\Prb(N_n\geq tn^{1/2})\to e^{-t^2/4}=\Prb(E>t^2/4)=\Prb(2\sqrt{E}>t)$ where $E\sim\Exp(1)$ has
an exponential distribution.
We also see that, for $0\le k\le n$,
\[
 \Prb(N_n\geq k)=\prod_{i=1}^{k-1}\left(1-\frac{i}{2n-i}\right)
 \ge 1-\sum_{i=1}^{k-1}\frac{i}{2n-i}
 \ge 1-\frac{k(k-1)/2}{2n-k}\ge 1-\frac{k^2}{2n}.
\]
Hence, as the above inequality is trivially true for $k>n$,
\[
 \Prb\big(N_n< \gamma n^{1/2}\big)=1-\Prb\big(N_n\geq \lfloor\gamma n^{1/2}\rfloor\big)
 \le \frac{\lfloor\gamma n^{1/2}\rfloor^2}{2n}\le \frac{\gamma^2}{2}.
 \qedhere
\]
\end{proof}

We use Lemma~\ref{lem:convtorayleigh} to prove the following lemma which considers the asymptotic value of
$\Prb(\pbac_1,\dots,\pbac_{N_n}\ge 0)=\E\big[\frac{(2N_n-1)!!}{2^{N_n}N_n!}\big]$.

\begin{lemma}\label{lem:probnonnegperturbed}
As $n\to\infty$,
\[
 \E\left[\frac{(2N_n-1)!!}{2^{N_n}N_n!}\right]=(1+o(1))n^{-1/4} \frac{\Gamma(3/4)}{\sqrt{2\pi}}.
\]
\end{lemma}

\begin{proof}
Fix some $\delta \in (0,1)$. We split the expectation according to the contribution
with $N_n<\delta n^{1/2}$ and $N_n\ge \delta n^{1/2}$.
By Lemma~\ref{lem:convtorayleigh}, $\Prb(N_n<\delta n^{1/2})\le \delta^2/2$ for all~$n$ and
we also recall from \eqref{e:ulbounds} that 
\[
 \frac{1}{\sqrt{\pi (m+1/2)}}\le \frac{(2m-1)!!}{2^m m!} \le \frac{1}{\sqrt{\pi m}}.
\]
Hence,
\begin{align*}
 n^{1/4}\E\left[\frac{(2N_n-1)!!}{2^{N_n}N_n!}\one_{\{N_n<\delta n^{1/2}\}}\right]
 &\le n^{1/4}\sum_{i=0}^\infty
  \Prb\big(2^{-i-1}\delta n^{1/2}\le N_n<2^{-i}\delta n^{1/2}\big)\frac{1}{\sqrt{\pi 2^{-i-1}\delta n^{1/2}}}\\
 &\le \sum_{i=0}^\infty \tfrac12(2^{-i}\delta)^2\cdot (\pi 2^{-i-1}\delta)^{-1/2}=O(\delta^{3/2}).
\end{align*}

Since $n^{1/4}/\sqrt{\pi N_n}$ is bounded for $N_n\ge \delta n^{1/2}$
and $n^{-1/2}N_n$ converges in distribution to $2\sqrt{E}$ where $E\sim\Exp(1)$,
\[
 n^{1/4}\E\left[\frac{(2N_n-1)!!}{2^{N_n}N_n!}\one_{\{N_n\ge \delta n^{1/2}\}}\right]
 =\E\left[\frac{n^{1/4}}{\sqrt{\pi N_n+O(1)}}\one_{\{n^{-1/2}N_n \ge \delta\}}\right]
 \to\E\left[\frac{\one_{\{2\sqrt{E}\ge \delta\}}}{\sqrt{2\pi\sqrt E}}\right]
\]
as $n\to\infty$. Now
\[
 \E\left[\frac{\one_{\{2\sqrt{E}\ge \delta\}}}{\sqrt{2\pi\sqrt E}}\right]
 =\E\left[\frac{\one_{\{E\ge \delta^2/4\}}}{\sqrt{2\pi}E^{1/4}}\right]
 =\frac{1}{\sqrt{2\pi}} \int_{\delta^2/4}^\infty x^{-1/4}e^{-x} dx
 =\frac{\Gamma(3/4)}{\sqrt{2\pi}}+O(\delta^{3/2}).
\]
Hence,
\[
 n^{1/4}\E\left[\frac{(2N_n-1)!!}{2^{N_n}N_n!}\right]
 \to \frac{\Gamma(3/4)}{\sqrt{2\pi}}+O(\delta^{3/2})
\]
as $n\to\infty$. As $\delta$ was arbitrary, the result now follows.
\end{proof}

We now want to show that
\[
 \Prb\left(\exists k \in (n^{1/2},n]: \lba_k=\lb_k=0 \mid \lba_1,\dots,\lba_n\ge 0\right)=o(1),
\]
which is the content of Lemma \ref{lem:hittingtimeat0}. For the proof of this, we will need Lemma \ref{lem:locallimitlazy}, which is a local limit theorem for the position and area of a lazy SSRW, and Lemma \ref{lem:upperboundintegral0}, which we will use to control the probability that the integral and position of an unconditioned lazy SSRW hit $0$ simultaneously at a late time. 

\begin{restatable}{lemma}{locallimitlazy}\label{lem:locallimitlazy}
We have 
\[
 \lim_{n\to \infty}\sup_{a,b}\,\left|n^2 \Prb\left(\lw_n=a, \lwa_n=b\right)-\phi(n^{-1/2}a,n^{-3/2}b)\right|=0
\]
where the supremum runs over all $(a,b)\in \Z^2$ and
\[
 \phi(x,y)=\frac{2\sqrt{3}}{\pi}\exp\big(-4x^2+12xy-12y^2\big).
\]
\end{restatable}
The proof of this lemma is postponed to Appendix \ref{app:locallimitlazy}. 

\begin{restatable}{lemma}{upperboundintegral}\label{lem:upperboundintegral0}
 There exists a constant $C$ such that for all\/~$n\ge1$,
 \[
 \Prb\left(\lwa_n=\lw_n=0,\,\lwa_1,\dots,\lwa_n\ge 0\right)\le Cn^{-5/2}.
 \]
\end{restatable}
With Lemma~\ref{lem:locallimitlazy} in hand, the proof of Lemma~\ref{lem:upperboundintegral0} is a direct
adaptation of the proof of the upper bound of Theorem~1 of \cite{AurzadaDereichLifshits2014} on simple
symmetric random walks. For the sake of completeness, we have included a proof in Appendix \ref{app:upperboundintegral}.

\begin{lemma}\label{lem:hittingtimeat0}
 We have that 
 \[
  \Prb\left(\exists k \in (n^{1/2},n]\colon \lba_k=\lb_k=0
  \mid \lba_1,\dots,\lba_n\ge 0\right)=O(n^{-3/4}).
 \]
\end{lemma}
\begin{proof}
Define $\probareapositive{n}=\Prb(\lwa_1,\dots,\lwa_n\ge 0,\,\lw_n=0)$. 
Note that for any $k$, 
\begin{align*}
 \Prb\left(\lba_k=\lb_k=0 \mid \lba_1,\dots,\lba_n\ge 0\right)
 &=\frac{1}{\probareapositive{n}}\Prb\left(\lwa_k=\lw_k=0,\,\lwa_1,\dots,\lwa_n\ge 0,\,\lw_n=0\right)\\
 &=\frac{1}{\probareapositive{n}}\Prb\left(\lwa_k=\lw_k=0,\,\lwa_1,\dots,\lwa_k\ge 0\right)\\
 &\quad \times \Prb\left(\lwa_{k+1},\dots,\lwa_n\ge 0,\,\lw_n=0 \mid \lwa_k=\lw_k=0\right)\\
 &=\frac{\probareapositive{n-k}}{\probareapositive{n}}\,\Prb\left(\lwa_k=\lw_k=0,\,\lwa_1,\dots,\lwa_k\ge 0\right). 
\end{align*}
By our earlier results (Proposition~\ref{prop:basic}), $\probareapositive{n}=\Theta(n^{-3/4})$ and by
Lemma~\ref{lem:upperboundintegral0}, the final factor is $O(k^{-5/2})$. Therefore, there exists a $C$
such that for each $k\le n$, 
\[
 \Prb\left(\lba_k=\lb_k=0 \mid \lba_1,\dots,\lba_n\ge 0\right)\le C k^{-5/2}(n-k+1)^{-3/4}n^{3/4}.
\]
Then, the result follows from the union bound by observing that
\[
 \sum_{n^{1/2}< k\leq n/2} Ck^{-5/2}(n-k+1)^{-3/4}n^{3/4}
 \le 2^{3/4}C\sum_{k> n^{1/2}} k^{-5/2}=O((n^{1/2})^{-3/2})=O(n^{-3/4})
\]
and
\[
 \sum_{k=\lfloor n/2\rfloor }^n Ck^{-5/2}(n-k+1)^{-3/4}n^{3/4}
 =O(n^{-7/4}\sum_{j=1}^{\lceil n/2\rceil +1} j^{-3/4})=O(n^{-7/4}n^{1/4})=O(n^{-3/2}).
 \qedhere
\]
\end{proof}

We now turn to Lemma~\ref{lem:pn_to_rho} which shows that $\probareazero{n}\to\rho$. 
For the proof of this we need the following result. 

\begin{lemma}\label{lem:ratioq_n}
 Let $q_n=\Prb(\lwa_1,\dots,\lwa_n\ge 0\mid \lw_n=0)$.
 Then, uniformly over all $m\le n^{1/2}$, we have $q_{n-m}/q_n\to 1$ as $n\to\infty$. 
\end{lemma}
\begin{proof}
We define $\varphi_\ell(k)=\Prb(\lw_\ell=-k)$, which is also the probability that a lazy SSRW starting at $k$ is at 0 at time $\ell$.
Let $0<m<b<n$ and let $Z$ be a random variable that depends only on
$\cF_{n-b}=\sigma(\lw_1,\dots,\lw_{n-b})$. We first show that
\begin{equation}\label{eq:req}
 \E[Z\mid \lw_n=0] = \E\left[Z\cdot\frac{\varphi_{b}(\lw_{n-b})}{\varphi_{b-m}(\lw_{n-b})} \frac{\varphi_{n-m}(0)}{\varphi_n(0)}\bigmid \lw_{n-m}=0\right].
\end{equation}
For $b'<n'$ and any $Z'$ that depends only on $\cF_{b'}$,
\begin{align*}
 \E[Z'\mid \lw_{n'}=0]
 &=\frac{\E[Z'\one{\{\lw_{n'}=0\}}]}{\Prb(\lw_{n'}=0)}\\
 &=\frac{\E[\E[Z'\one{\{\lw_{n'}=0\}}\mid \cF_{b'}]]}{\Prb(\lw_{n'}=0)}\\
 &=\frac{\E[Z'\Prb(\lw_{n'}=0\mid \cF_{b'})]}
 {\Prb(\lw_{n'}=0)}
 =\E\left[Z'\cdot \frac{\varphi_{n'-b'}(\lw_{b'})}{\varphi_{n'}(0)}\right].
\end{align*}
Applying this to $Z'=Z \frac{\varphi_{b}(\lw_{n-b})}{\varphi_{b-m}(\lw_{n-b})} \frac{\varphi_{n-m}(0)}{\varphi_n(0)}$, $n'=n-m$ and $b'=n-b$, we find that
\begin{align*}
 \E[Z'\mid \lw_{n-m}=0]
 &=\E\left[Z'\cdot\frac{\varphi_{b-m}(\lw_{n-b})}{\varphi_{n-m}(0)}\right]\\
 &=\E\left[Z\cdot\frac{\varphi_{b}(\lw_{n-b})}{\varphi_{b-m}(\lw_{n-b})} \frac{\varphi_{n-m}(0)}{\varphi_n(0)} \cdot \frac{\varphi_{b-m}(\lw_{n-b})}{\varphi_{n-m}(0)}\right]\\
 &=\E\left[Z\cdot\frac{\varphi_{b}(\lw_{n-b})}{\varphi_n(0)} \right]\\
 &=\E[Z\mid \lw_n=0]
\end{align*}
as desired. 

We assume $m\le n^{1/2}$ and let $b = b(n) \in \mathbb{N}$ be such that $b(n) =(1+o(1)) n^{7/9}$. 
Next, we provide bounds on
$\frac{\varphi_{b}(\lw_{n-b})}{\varphi_{b-m}(\lw_{n-b})}\frac{\varphi_{n-m}(0)}{\varphi_n(0)}$.
We first note that by \eqref{eq:prb_Y} for $|k|\le\ell$ we have
\begin{align*}
 \varphi_\ell(k)
 &=\binom{2\ell}{\ell-k}4^{-\ell}=\binom{2\ell}{\ell}4^{-\ell}\cdot\prod_{j=1}^k\Big(1-\frac{2j-1}{\ell+j}\Big)\\
 &=\frac{1}{\sqrt{\pi\ell+O(1)}}\cdot\exp\left(\sum_{j=1}^k\log\Big(1-\frac{2j-1}{\ell+j}\Big)\right)\\
 &=\frac{1}{\sqrt{\pi\ell}}e^{O(1/\ell)}\cdot\exp\left(-\sum_{j=1}^k\Big(\frac{2j-1}{\ell}+O(j^2/\ell^2)\Big)\right)\\
 &=\frac{1}{\sqrt{\pi\ell}}e^{-k^2/\ell+O(k^3/\ell^2)+O(1/\ell)}.
\end{align*}
Hence, $\varphi_\ell(k)/\varphi_{\ell'}(k)\to 1$ provided $\ell\to\infty$,
$\ell/\ell'\to 1$, $k^2(\ell-\ell')/\ell\ell'\to 0$ and $k^3/\ell^2\to0$.
Now define
\begin{align*}
 \underline{\delta}_n&=\min_{|k|\le n^{1/2}}
  \frac{\varphi_{b}(k)}{\varphi_{b-m}(k)} \frac{\varphi_{n-m}(0)}{\varphi_n(0)}\\
 \overline{\delta}_n&=\max_{|k|\le n^{1/2}},
  \frac{\varphi_{b}(k)}{\varphi_{b-m}(k)} \frac{\varphi_{n-m}(0)}{\varphi_n(0)}.
\end{align*}
and note that $\underline{\delta}_n\to 1$ and $\overline{\delta}_n\to 1$ as $n\to \infty$
for our choice of $m$ and~$b$. Indeed, 
$\log\underline{\delta}_n,\log\overline{\delta}_n=O(m/b+mn/b^2+n^{3/2}/b^2)\to0$.

Applying \eqref{eq:req} with $Z=\one_{\{\lwa_1,\dots,\lwa_{n-b}\ge0,\,|\lw_{n-b}|\le n^{1/2}\}}$
gives
\begin{align*}
\Prb\big(\lwa_1,\dots,\lwa_{n-b}\ge 0\mid \lw_n=0\big)
&\le \Prb\big(\lwa_1,\dots,\lwa_{n-b}\ge 0,\,|\lw_{n-b}|\le n^{1/2}\mid \lw_n=0\big)\\
&\quad+\Prb(|\lw_{n-b}|>n^{1/2}\mid \lw_n=0)\\
&\le \overline{\delta}_n\Prb\big(\lwa_1,\dots,\lwa_{n-b}\ge 0,\,|Y_{n-b}|\le n^{1/2}\mid \lw_{n-m}=0\big)\\
&\quad+\Prb(|\lw_{n-b}|>n^{1/2}\mid \lw_n=0)\\
&\le \overline{\delta}_n\Prb\big(\lwa_1,\dots,\lwa_{n-b}\ge 0\mid \lw_{n-m}=0\big)\\
&\quad+\Prb(|\lw_{n-b}|>n^{1/2}\mid \lw_n=0).
\end{align*}

Similarly,
\begin{align*}
\Prb\big(\lwa_1,\dots,\lwa_{n-b}\ge 0\mid \lw_n=0\big)
&\ge \Prb\big(\lwa_1,\dots,\lwa_{n-b}\ge 0,\,|\lw_{n-b}|\le n^{1/2}\mid \lw_n=0\big)\\
&\ge \underline{\delta}_n\Prb\big(\lwa_1,\dots,\lwa_{n-b}\ge 0,\,|Y_{n-b}|\le n^{1/2}\mid \lw_{n-m}=0\big)\\
&\ge \underline{\delta}_n\Prb\big(\lwa_1,\dots,\lwa_{n-b}\ge 0\mid \lw_{n-m}=0\big)\\
&\quad-\underline{\delta}_n\Prb(|\lw_{n-b}|>n^{1/2}\mid \lw_{n-m}=0).
\end{align*}

We note that
\begin{align*}
 &\Prb(|\lw_{n-b}|>n^{1/2}\mid \lw_n=0)=\Prb(|\lw_{b}|>n^{1/2} \mid \lw_n=0)=O(n^{-\omega(1)}),\\
 &\Prb(|\lw_{n-b}|>n^{1/2}\mid \lw_{n-m}=0)=\Prb(|\lw_{b-m}|>n^{1/2} \mid \lw_{n-m}=0) =O(n^{-\omega(1)})
\end{align*}
uniformly in all $m\le n^{1/2}$, where we use that on the event $\{\lw_n=0\}$,
it holds that $(\lw_{i})_{i=1}^n$ and $(\lw_{n-i})_{i=1}^n$ have the same law and the fact
that $n^{1/2}=(1+o(1)) n^{\eps}b^{1/2}$ for some $\eps>0$ (and large enough $n$). 
Therefore, using that $q_n=\Theta(n^{-1/4})$ by Proposition \ref{prop:basic},
\begin{equation}\label{eq:fromnton-m}
 \Prb\big(\lwa_1,\dots,\lwa_{n-b}\ge 0\mid \lw_n=0\big)
 =(1+o(1))\Prb\big(\lwa_1,\dots,\lwa_{n-b}\ge 0\mid \lw_{n-m}=0\big).
\end{equation}
 
To finish the proof, it suffices to show that 
\begin{equation}\label{eq:nosignchangeatend}
 \frac{\Prb(\lwa_1,\dots,\lwa_{n-b}\ge 0\mid \lw_n=0)}{\Prb( \lwa_1,\dots,\lwa_n\ge0\mid \lw_n=0)}\to 1
\end{equation}
as $n\to \infty$. Indeed, if we set $b = \lfloor n^{7/9} \rfloor$, we see that
\begin{align*}
 \frac{q_{n-m}}{q_n}
 &=\frac{\Prb(\lwa_1,\dots,\lwa_{n-m}\ge 0\mid \lw_{n-m}=0)}{\Prb(\lwa_1,\dots,\lwa_n\ge 0\mid \lw_n=0)}\\
 &=\frac{\Prb(\lwa_1,\dots,\lwa_{n-m}\ge 0\mid \lw_{n-m}=0)}{\Prb(\lwa_1,\dots,\lwa_{n-b}\ge 0\mid \lw_{n-m}=0)}
   \frac{\Prb(\lwa_1,\dots,\lwa_{n-b}\ge 0\mid \lw_{n-m}=0)}{\Prb(\lwa_1,\dots,\lwa_{n-b}\ge 0\mid \lw_{n}=0)}\\
 &\quad\times\frac{\Prb(\lwa_1,\dots,\lwa_{n-b}\ge 0\mid \lw_{n}=0)}{\Prb(\lwa_1,\dots,\lwa_n\ge 0\mid \lw_n=0)},
\end{align*}
so \eqref{eq:nosignchangeatend} would imply that the first and the third factor in the product tend to~$1$.
The second factor tends to $1$ by~\eqref{eq:fromnton-m}. 

To obtain \eqref{eq:nosignchangeatend} we need to show that 
\[
 \Prb\big(\exists k\in(n-b,n]\colon \lba_k< 0 \mid \lba_1,\dots,\lba_{n-b}\ge 0\big)\to 0
\]
as $n\to \infty$. We know by our earlier bounds (Proposition~\ref{prop:basic}) that there
is a $c>0$ such that for all $n$ large enough
\[
 \Prb\left(\lba_1,\dots,\lba_{n}\ge 0 \right)\ge  cn^{-1/4},
\]
so it is sufficient to show that 
\[
 \Prb\left(\exists k\in(n-b,n]\colon \sgn(\lba_{k-1})\ne \sgn(\lba_k) \right)=o(n^{-1/4}).
\]
Note that 
\begin{align*}
 \Prb\big(\exists k\in(n-b,n]\colon& \sgn(\lba_{k-1})\ne\sgn(\lba_k)\big)\\
 &\le \Prb\big(|\lba_{n}|\le n^{6/5}\big)+\Prb\Big(\max_{n-b+1\le i\le n}|\lb_i| > n^{6/5}/b\Big)\\
 &\le \Prb\big(|\lba_{n}|\le n^{6/5}\big)+\Prb\Big(\max_{1\le i\le b}|\lb_i| > n^{6/5}/b\Big)
\end{align*}
by the union bound and the fact that $(\lb_{i})_{1\le i\le n}$ and $(\lb_{n+1-i})_{1\le i\le n}$ have the same law.
By Lemma \ref{lem:locallimitlazy},  
\[
 \Prb\big(|\lba_{n}|\le n^{6/5}\big)=\frac{\Prb(|\lwa_n|\le n^{6/5},\,Y_n=0)}{\Prb(Y_n=0)}
 =\frac{O(n^{6/5}/n^2)}{\Theta(n^{-1/2})}=o(n^{-1/4}).
\]

Furthermore, by the reflection principle for the simple symmetric random walk~$\uw$, we have that, for any $k$
\[
 \Prb\Big(\max_{1\le i\le 2b}\uw_i\ge k\Big)\le 2\Prb\Big(\uw_{2b}\ge k\Big).
\]
Under the usual coupling between $(\lw_k)_{1\le k\le b}$ and $(\uw_k)_{1\le k\le 2b}$ we have that 
\[
 \Prb\Big(\max_{1\le i\le b}|\lw_i|>k\Big)=2\Prb\Big(\max_{1\le i\le b}\lw_i>k\Big)
 \le 2\Prb\Big(\max_{1\le i\le 2b}\uw_i\ge 2k\Big)\le 4\Prb\Big(\uw_{2b}\ge 2k\Big),
\]
so observing that $n^{6/5}/b=n^{\eps}b^{1/2}$ for some $\eps>0$ implies that
\[
 \Prb\Big(\max_{1\le i\le b}|\lw_i| > n^{6/5}/b\Big)=n^{-\omega(1)}.
\]
Then, $\Prb(\lw_n=0)=\Theta(n^{-1/2})$ implies that also 
\[
 \Prb\Big(\max_{1\le i \le b}|\lb_i| > n^{6/5}/b\Big)=n^{-\omega(1)}.
\]
So it follows that 
\[
 \Prb\left(\exists k\in(n-b,n]\colon \sgn(\lba_{k-1})\ne \sgn(\lba_k)\right)=o(n^{-1/4})
\]
as claimed.
\end{proof}
Recall that $\probareazero{n}=\Prb(M_n\ge 1\mid \lbac_1,\dots,\lbac_{N_n}\ge 0)$ where $M_n=\#\{i>0:\lbac_i=0\}$.
We will use the preceding lemma to prove Lemma~\ref{lem:pn_to_rho}.

\begin{lemma}\label{lem:pn_to_rho}
We have $\probareazero{n}\to \rho$ as $n\to \infty$.
\end{lemma}
\begin{proof}
By definition, 
\[
 \probareazero{n}=\Prb(\exists k\in [n]\colon \lba_k=\lb_k=0 \mid \lba_1,\dots,\lba_{n}\ge 0).
\]
Let $\zeta_1=\zeta_1(\lb)=\min\{k>0:\lb_k=0,\,\lba_k\le 0\}$, with the convention that $\min\emptyset=\infty$. On $\{\lba_1,\dots,\lba_n\ge 0\}$ it holds that $\{\exists k\in [n]\colon\lba_k=\lb_k=0\}=\{\zeta_1\le n\}$ and we first show that, under this conditioning, with high probability $\zeta_1$ is either at most $n^{1/2}$ or larger than $n$. 
We will then show that we may stop conditioning on $\{\lba_1,\dots,\lba_n\ge 0\}$ if we instead consider the event $\{\zeta_1\le n^{1/2},\,\lba_{\zeta_1}=0\}$ (and accept a $o(1)$ term). We then further show that the fact that we are considering $Y^{br}$ instead of $Y$ makes a negligible difference. After making these changes, the random variable $\zeta_1$ does not depend on $n$ and the result follows.

Observe that
\begin{align*}
 \probareazero{n}&=\Prb(\zeta_1\le n \mid \lba_1,\dots,\lba_n\ge 0)\\
 &=\Prb(\zeta_1\le n^{1/2} \mid \lba_1,\dots,\lba_n\ge 0)
  +\Prb(\zeta_1\in (n^{1/2},n] \mid \lba_1,\dots,\lba_n\ge 0)
\end{align*}
 and by Lemma~\ref{lem:hittingtimeat0},
\begin{multline*}
 \Prb(\zeta_1\in (n^{1/2},n] \mid \lba_1,\dots,\lba_n\ge 0)\\
 \le \Prb\left(\exists k \in (n^{1/2},n]\colon \lba_k=\lb_k=0 \mid \lba_1,\dots,\lba_n\ge 0\right)
 =O(n^{-3/4}).
\end{multline*}
It remains to show that $\Prb(\zeta_1\le n^{1/2} \mid \lba_1,\dots,\lba_n\ge 0)\to \rho$.
Recall that $\probbridgeareapositive{n}=\Prb(\lba_1,\dots,\lba_n\ge 0)$. Then, for any $k\in [n]$,
\begin{align*}
 \Prb(\lba_1,\dots,\lba_n\ge 0\mid \zeta_1=k)
 &=\Prb(\lba_1,\dots,\lba_k\ge 0\mid \zeta_1=k)\\
 &\quad\times\Prb(\lba_{k+1},\dots,\lba_{n}\ge 0\mid \zeta_1 =k,\,\lba_{k}=0)\\
 &=\probbridgeareapositive{n-k}\Prb(\lba_1,\dots,\lba_k\ge 0\mid \zeta_1=k)\\
  &=\probbridgeareapositive{n-k}\Prb(\lba_{\zeta_1}=0\mid \zeta_1=k).
\end{align*}
where we used that, on the event $\{\zeta_1=k,\lba_k=0\}$, the restrictions of $\lb$ and $\lba$ to $[k,n]$ 
have the joint law of a lazy simple symmetric random walk bridge on $[n-k]$ and its area process respectively. 
Therefore,
\begin{align*}
 \Prb(\zeta_1\le n^{1/2}\mid \lba_1,\dots,\lba_n\ge 0)
 =\sum_{k=1}^{\lfloor n^{1/2} \rfloor}\Prb(\zeta_1=k)
 \frac{\probbridgeareapositive{n-k}}{\probbridgeareapositive{n}}\Prb(\lba_{\zeta_1}=0\mid \zeta_1=k).
\end{align*}
Lemma~\ref{lem:ratioq_n} shows that $\frac{\probbridgeareapositive{n-k}}{\probbridgeareapositive{n}}\to 1$ as $n\to \infty$
uniformly over all $k\le n^{1/2}$, so
\begin{align*}
 \probareazero{n}
 &=\sum_{k=1}^{\lfloor n^{1/2}\rfloor}\Prb(\zeta_1=k)\Prb\left(\lba_{\zeta_1}=0 \mid \zeta_1=k \right)+o(1)\\
 &=\Prb\left(\zeta_1\le n^{1/2},\,\lba_{\zeta_1}=0\right)+o(1).
\end{align*}
We now show that removing the condition that $\{\lw_n=0\}$ has only a negligible effect on the probability.
Fix $0<\eps<1/4$ and note that 
\begin{align*}
 \Prb\big(\zeta_1\le n^{1/2},\ &\lba_{\zeta_1}=0\big)\\
 &=\Prb\left(\zeta_1\le n^{1/2},\lba_{\zeta_1}=0,|\lb_{\lfloor n^{1/2} \rfloor}|<n^{1/4+\eps}\right)+o(1)\\
 &=\E\left[\one_{\left\{\zeta_1\le n^{1/2}, \lwa_{\zeta_1}=0, |\lw_{\lfloor n^{1/2}\rfloor}|<n^{1/4+\eps}\right\}}
  \frac{\varphi_{n-\lfloor n^{1/2}\rfloor}(\lw_{\lfloor n^{1/2} \rfloor})}{\varphi_{n}(0)}\right]+o(1),
\end{align*}
where we have used the function $\varphi$ from the proof of Lemma~\ref{lem:ratioq_n} and that the event
\[
 \left\{\zeta_1\le n^{1/2}, \lba_{\zeta_1}=0, |\lb_{\lfloor n^{1/2} \rfloor}|<n^{1/4+\eps} \right\}
\]
is measurable with respect to $\sigma(\lb_1,\dots,\lb_{\lfloor n^{1/2}\rfloor})$. 
From the proof of Lemma~\ref{lem:ratioq_n} we have that
\[
 \frac{\varphi_{n-\lfloor n^{1/2}\rfloor}(a)}{\varphi_{n}(0)}\to 1
\]
uniformly over all $|a|<n^{1/4+\eps}$, so 
\begin{align*}
 \Prb\left(\zeta_1\le n^{1/2}, \lba_{\zeta_1}=0 \right)
 &=\Prb\left(\zeta_1\le n^{1/2}, \lwa_{\zeta_1}=0, |\lw_{\lfloor n^{1/2} \rfloor}|<n^{1/4+\eps}\right)+o(1)\\
 &=\Prb\left(\zeta_1\le n^{1/2}, \lwa_{\zeta_1}=0\right)+o(1).
\end{align*}
Now observe  that, under the law of~$\lw$, we have that $\zeta_1$ is a random variable on $\N$
that does not depend on~$n$, so $\Prb\left(\zeta_1\le n^{1/2},\lwa_{\zeta_1}=0\right)\to \rho$
as $n\to \infty$ and the statement follows.
\end{proof}

\subsection{Conditioning on ending in \texorpdfstring{$0$}{0} or \texorpdfstring{$-1$}{-1}}\label{subsec:squaretorectangle}
In this section, we show how Proposition~\ref{prop:areanonnegrectangle} follows
from Proposition~\ref{prop:areanonneg}. By Lemma~\ref{lem:to_walk}, we need to show
that for $(\lw_k)_{k\ge1}$ a lazy SSRW, and $\lwa_k=\sum_{i=1}^k\lw_i$, we have that
\[
 \Prb\left(\lwa_1,\dots,\lwa_{n-1}\ge 0 \mid \lw_{n-1} \in \{0, -1\}\right)=(1+o(1)) n^{-1/4}\frac{\Gamma(3/4)}{\sqrt{2\pi(1-\rho)}}.
\]

\begin{proof}[Proof of Proposition~\ref{prop:areanonnegrectangle}]
First, observe that by Proposition~\ref{prop:areanonneg} and the fact that
$\Prb(\lw_n=0)=(1+o(1)) \tfrac{1}{\sqrt{\pi n}}$,
\[
 \Prb\left(\lwa_1,\dots,\lwa_{n-1}\ge 0,\,\lw_{n-1}=0\right)
 =(1+o(1)) n^{-3/4}\frac{\Gamma(3/4)}{\pi\sqrt{2(1-\rho)}}.
\]
We also need to calculate
\[
 \Prb\left(\lwa_1,\dots,\lwa_{n-1}\ge 0,\,\lw_{n-1}=-1\right).
\]
Note that 
\begin{align}
 \Prb(\lwa_1,\dots,\lwa_n \ge 0,\lw_{n}=0)
 &=\tfrac{1}{4}\Prb(\lwa_1,\dots,\lwa_{n-1}\ge 0,\lw_{n-1}=-1)\notag\\
  &\quad+ \tfrac{1}{4}\Prb(\lwa_1,\dots,\lwa_{n-1}\ge 0,\lw_{n-1}=1)\notag\\
 &\quad+\tfrac{1}{2}\Prb(\lwa_1,\dots,\lwa_{n-1}\ge 0,\lw_{n-1}=0).\label{eq:splitpositionatlaststep}
\end{align}
Both $\Prb(\lwa_1,\dots,\lwa_n\ge 0,\lw_{n}=0)$ and $\Prb(\lwa_1,\dots,\lwa_{n-1}\ge 0,\lw_{n-1}=0)$
are asymptotically equal to $n^{-3/4}\frac{\Gamma(3/4)}{\pi\sqrt{2(1-\rho)}}$, so if we show that
\begin{equation}\label{eq:equalprobpositivenegative}
 \Prb(\lwa_1,\dots,\lwa_{n-1}\ge 0,\lw_{n-1} =-1) =(1+o(1))
 \Prb(\lwa_1,\dots,\lwa_{n-1}\ge 0,\lw_{n-1} =1),
\end{equation}
we can deduce that
\[
 \Prb(\lwa_1,\dots,\lwa_{n-1}\ge 0,\lw_{n-1} =-1)=(1+o(1)) n^{-3/4}\frac{\Gamma(3/4)}{\pi\sqrt{2(1-\rho)}}.
\]
Then, the fact that $\Prb(\lw_{n-1}\in\{0,-1\})=(1+o(1))\tfrac{2}{\sqrt{\pi n}}$ implies that 
\[
 \Prb(\lwa_1,\dots,\lwa_{n-1}\ge 0 \mid \lw_{n-1}\in\{0, -1\})
 =(1+o(1)) n^{-1/4}\frac{\Gamma(3/4)}{\sqrt{2\pi(1-\rho)}},
\]
as required.

We now prove \eqref{eq:equalprobpositivenegative}.
Let $\cY_-$ be the set of paths $(y_0,y_1,\dots,y_{n-1})$ starting at $0$ with steps in
$\{-1,0,1\}$ which satisfy 
\[
 \left\{\sum_{i=1}^k y_i\ge 0\text{ for all }k\in [n-1],\,y_{n-1} =-1\right\},
\]
and let $\cY_+$ be the set of paths $(y_0,y_1,\dots,y_{n-1})$ starting at $0$ with steps in
$\{-1,0,1\}$ which satisfy  
\[
 \left\{\sum_{i=1}^k y_i\ge 0\text{ for all }k\in [n-1],\,y_{n-1} =1\right\}.
\]
We have
\[
 \Prb(\lwa_1,\dots,\lwa_{n-1}\ge 0, \lw_{n-1}=-1)=\Prb(Y\in \cY_-)
\]
and
\[
 \Prb(\lwa_1,\dots,\lwa_{n-1}\ge 0, \lw_{n-1}=1)=\Prb(Y\in \cY_+).
\]
We will define an injective map $f$ from $\cY_-$ to $\cY_+$
such that $\Prb(Y=y)=\Prb(Y=f(y))$ for all $y\in \cY_-$. Then, 
\begin{align*}\Prb\left(Y\in \cY_-\right)&=\sum_{y\in \cY_-} \Prb(Y=y)\\
&=\sum_{y\in \cY_-} \Prb(Y=f(y))\\&=\sum_{y\in f(\cY_-)} \Prb(Y=y)\\&=
\Prb(Y\in \cY_+)-\Prb(Y\in \cY_+\setminus f(\cY_-)).\end{align*}
so we will be done if we can show that $\Prb(Y\in \cY_+\setminus f(\cY_-))=o(n^{-3/4})$. 
We let $f$ be the map that changes the sign of the last excursion away from $0$. To be precise, for $y\in \cY_-$, let $\tau_{max}=\tau_{max}(y)=\max\{k\le n-1:y_k=0\}$ and let
\[
 f(y)_i=\begin{cases}y_i,&\text{if }i\le\tau_{max};\\-y_i,&\text{if }i>\tau_{max}.\end{cases}
\]
It is immediate that $f$ has the claimed properties. In particular, $f(\cY_-)\subseteq \cY_+$ because for $y\in \cY_-$, it holds that $f(y)_k\ge y_k$ for all $k$.
Moreover, $\cY_+\setminus f(\cY_-)$
consists of the paths $y$ that end at $1$ for which the area at time $n-1$ is smaller than twice the area of the last excursion away from $0$, that is,
\[
 \Prb(Y\in \cY_+, Y\notin f(\cY_-))
 =\Prb\left(\sum_{i=1}^k \lw_i\ge 0\text{ for }k\in [n-1],\,\lw_{n-1} =1,
 \sum_{i=1}^{n-1}\lw_i<2\sum_{i=\tau_{max}}^{n-1}\lw_i\right).
\]
We see that 
\begin{multline*}
 \Prb\left(\sum_{i=1}^k \lw_i\ge 0\text{ for }k\in [n-1],\,\lw_{n-1} =1,
 \sum_{i=1}^{n-1} \lw_i<2\sum_{i=\tau_{max}}^{n-1}\lw_i\right)\\
 \le 4\Prb\left(\sum_{i=1}^k \lw_i\ge 0\text{ for }k\in [n],\,\lw_{n} =0,
 \sum_{i=1}^{n} \lw_i<2\sum_{i=\tau_{max}}^{n}\lw_i\right),
\end{multline*}
by conditioning on the value of the $n$th increment (similar to the calculation \eqref{eq:splitpositionatlaststep})
and 
\begin{multline*}
 \Prb\left(\sum_{i=1}^k \lw_i\ge 0\text{ for }k\in [n],\,\lw_{n} =0,
 \sum_{i=1}^{n} \lw_i<2\sum_{i=\tau_{max}}^{n}\lw_i\right)\\
 \le \Prb\left(0\le \sum_{i=1}^{n} \lw_i<2\sum_{i=\tau_{max}}^{n}\lw_i \bigmid \lw_{n} =0\right)\Prb(\lw_n=0).
\end{multline*}
Since $\Prb(\lw_n=0)=\Theta(n^{-1/2})$, we are done if we can show that 
\[
 \Prb\left(0\le \sum_{i=1}^{n} \lb_i<2\sum_{i=\tau_{max}}^{n}\lb_i\right)=o(n^{-1/4}).
\]
Pick any $0<\eps<1/4$. We compute
\begin{align*}
 \Prb\left(0\le \sum_{i=1}^{n} \lb_i<2\sum_{i=\tau_{max}}^{n}\lb_i\right)
 &\le \Prb\left(\sum_{i=\tau_{max}}^{n}\lb_i\ge  n^{1+\eps}\right)
 +\Prb\left(0\le \sum_{i=1}^{n} \lb_i<2 n^{1+\eps}\right)\\
 &=\Prb\left(\sum_{i=1}^{\tau_1}\lb_i\ge n^{1+\eps}\right)
 +\Prb\left(0\le \sum_{i=1}^{n} \lb_i<2 n^{1+\eps}\right),
\end{align*}
where $\tau_1=\tau_1(\lb)$ is the first return time of $\lb$ to~$0$.
We trivially have that $\sum_{i=1}^{\tau_1}\lb_i\le (\tau_1)^2$, and it is easy to see that  
\[
 \Prb(\tau_1=k)=\frac{\tfrac{1}{2k-1}\binom{2k-1}{k}\binom{2n-2k}{n-k} }{\binom{2n}{n} }
 =(1+o(1)) \tfrac{1}{4\sqrt{\pi}} k^{-3/2}n^{1/2}(n-k)^{-1/2},
\]
as $k,n \to \infty$. Hence,
\[\Prb\left(\tau_1\ge n^{1/2+\eps/2}\right)=O(n^{-1/4-\eps/4})=o(n^{-1/4}). \]
Moreover, by Lemma \ref{lem:locallimitlazy}, 
we see that 
\[\Prb\left(0\le \sum_{i=1}^{n} \lb_i<2 n^{1+\eps}\right)=O(n^{1+\eps}n^{-3/2})=o(n^{-1/4}),\]
so we conclude that 
\[\Prb(Y\in \cY_+, Y\not\in f(\cY_-))=o(n^{-3/4}).\]
This proves \eqref{eq:splitpositionatlaststep} and
\[
 \Prb\left(\sum_{i=1}^k \lw_i\ge 0\text{ for all }k\in [n],\,\lw_{n} =0\right)
 =(1+o(1)) n^{-3/4}\frac{\Gamma(3/4)}{\pi\sqrt{2(1-\rho)}}. \qedhere
\]
\end{proof}

\subsection{About half of the sums are even}
\label{subsec:half}
So far we have looked at the probability that a sequence $n-1\ge d_1\ge\dotsb\ge d_n\ge 0$
satisfies the dominating condition~\eqref{eq:DC}. In this section, we consider the parity condition that $d_1 + \dotsb +  d_n$ is even and prove Lemma~\ref{lem:parity}. 

Among the sequences $n-1\ge d_1\ge\dotsb\ge d_n\ge 0$ that satisfy the dominating
condition~\eqref{eq:DC}, let $\mathcal{E}$ denote the set of such sequences for which the sum $d_1 + \dotsb +  d_n$ is even, and $\mathcal{O}$ the sequences for which the sum is odd. We will define a partial matching between $\mathcal{E}$ and $\mathcal{O}$
such that the number of unmatched elements of $ \mathcal{E}\cup  \mathcal{O}$ is $o(| \mathcal{E}|+| \mathcal{O}|)$.
This will immediately imply Lemma~\ref{lem:parity}.

Each sequence $n-1\ge d_1\ge\dotsb\ge d_n\ge 0$ corresponds to a right/down path from
$(0,n-1)$ to $(n,0)$ taking $2n-1$ steps. This corresponds to a sequence $(B_1,\dots,B_{2n-1})$
where $B_i\in\{\rightarrow,\downarrow\}$ takes the value $\rightarrow$ if the path goes right
(which it does $n$ times) and $\downarrow$ if the path goes down (which it does $n-1$ times). 
We say $j\in [2n-2]$ with $j\equiv 0\mod 2$ is a \emph{switch position} for the sequence $(B_1,\dots,B_{2n-1})$ if
$(B_j,B_{j+1})\in
\{(\rightarrow,\downarrow),(\downarrow,\rightarrow)\}$. 
Switching the sequence $(B_1,\dots,B_{2n-1})$ at position $j$ refers to replacing
$(B_j,B_{j+1})$ with the unique element in
$\{(\rightarrow,\downarrow),(\downarrow,\rightarrow)\}
\setminus \{(B_j,B_{j+1})\}$, resulting in some new sequence $(B'_1,\dots,B'_{2n-1})$.
We make two observations:
\begin{itemize}
 \item First, for any $j\in [2n-2]$ with $j\equiv 0\bmod 2$, the position $j$ is a switch position for the
  sequence $B$ if and only if $j$ is a switch position for the sequence $B'$. Moreover, switching at position $j$ is self-inverse:
  switching $B'$ at position $j$ gives $B$.
 \item If $B$ corresponds to a sequence $n-1\ge d_1\ge \dotsb \ge d_n\ge 0$, then the sequence
  $B'$ also corresponds to a sequence $n-1\ge d'_1\ge \dotsb \ge d'_n\ge 0$, where for some $k\in [n]$
  \[
   d'_i\in\begin{cases}\{d_i\},&\text{if }i\ne k;\\ \{d_i-1,d_i+1\},&\text{if }i=k.\end{cases}
  \]
  In particular, the parities of $\sum_i d_i$ and $\sum_i d'_i$ are different.
\end{itemize}
It is not necessarily the case that performing a switch on a sequence in $\mathcal{E}$ results in a sequence
in $\mathcal{O}$ (or vice versa) as the switched sequence may violate the dominating condition~\eqref{eq:DC}.
Therefore, we will only define the matching between subsets $ \mathcal{E}' \subseteq  \mathcal{E}$ and $ \mathcal{O}' \subseteq \mathcal{O}$, where we choose $ \mathcal{E}'$ and $ \mathcal{O}'$ so that for all sequences in $ \mathcal{E}'\cup \mathcal{O}'$, have some slack
in the dominating condition~\eqref{eq:DC}. We will show that $| \mathcal{E}'\cup \mathcal{O}'|=(1+o(1))| \mathcal{E}\cup \mathcal{O}|$, for which we need the following two lemmas.

\begin{lemma}\label{lem:slack}
 As $n\to \infty$,
 \[
  \Prb(\lwa_{\lfloor n/2 \rfloor},\dots,\lwa_{n-1}\ge 1
   \mid \lwa_1,\dots,\lwa_{n-1}\ge 0, \lw_{n-1}\in\{-1,0\})=1-o(1).
 \]
\end{lemma}
\begin{proof}
By Proposition \ref{prop:areanonnegrectangle}, $\Prb(\lwa_1,\dots,\lwa_{n-1}\ge 0, \lw_{n-1}\in\{-1,0\})=\Theta(n^{-3/4})$, so we are done if we show that 
\[\Prb(\lwa_i=0\text{ for some }i\geq \lfloor n/2\rfloor, \lwa_1,\dots,\lwa_{n-1}\ge 0, \lw_{n-1}\in\{-1,0\})=o(n^{-3/4}).\]
But by conditioning on the value of the $n$th increment, we see that 
\begin{align*}&\Prb(\lwa_i=0\text{ for some }i\geq \lfloor n/2\rfloor, \lwa_1,\dots,\lwa_{n-1}\ge 0, \lw_{n-1}\in\{-1,0\})\\&\quad \leq 4 \Prb(\lwa_i=0\text{ for some }i\geq \lfloor n/2\rfloor, \lwa_1,\dots,\lwa_{n-1}\ge 0, \lw_{n}=0),\end{align*}
and Lemma \ref{lem:hittingtimeat0} and Proposition \ref{prop:areanonneg} imply that the right-hand side is $o(n^{-3/4})$ as claimed.
\end{proof}
Informally, this lemma states that switching the sequence `near the last checks' is unlikely to affect whether the  dominating condition (\ref{eq:DC}) holds. 
We next show that all but a negligible fraction of the sequences have switch positions `near the last checks' (which is the near the middle of $B$). 

We say a sequence $(B_1,\dots,B_{2n-1})$ is $k$-\emph{switchable} if it has a switch position at some
even $i\in [n-2k,n-2]$. 
\begin{lemma}\label{lem:switches_everywhere}
 For any $k\in[n-1]$, the number of sequences $(B_1,\dots,B_{2n-1})\in \{\rightarrow,\downarrow\}^{2n-1}$
 that are not $k$-switchable is at most $2^{-k}4^n$.
\end{lemma}
\begin{proof}
We choose a sequence $(B_1,\dots,B_{2n-1}) \in \{\rightarrow,\downarrow\}^{2n-1}$ by independently choosing each $B_i$ uniformly at random. The probability that $B$ has a switch at an even position $j\in [2n-1]$ is $\frac12$, and the events are independent for different even values of $j$. The probability that there are $k$ independent failures is hence $2^{-k}$.
\end{proof}
We now have all the necessary set-up to conclude the proof.
\begin{proof}[Proof of Lemma \ref{lem:parity}]
Set $k=\lfloor n/4\rfloor$. Let $ \mathcal{F}$ denote the set of sequences $d\in  \mathcal{E}\cup \mathcal{O}$ for which the right/down-sequence $(B_1,\dots,B_{2n-1})$ that corresponds it (as defined earlier this section) is not $k$-switchable,
or for which the area process takes the value $0$ at some point in the last $2k$ steps.
Then, applying Lemma~\ref{lem:slack} and Lemma~\ref{lem:switches_everywhere}, we find that $| \mathcal{F}|=o(| \mathcal{E}\cup  \mathcal{O}|)$. 

We set $ \mathcal{O}'= \mathcal{O}\setminus  \mathcal{F}$ and $ \mathcal{E}'= \mathcal{E}\setminus  \mathcal{F}$ and define a matching between
$ \mathcal{O}'$ and $ \mathcal{E}'$ as follows. Let $d\in  \mathcal{E}'\cup  \mathcal{O}'$ and let $B\in \{\rightarrow,\downarrow\}^{2n-1}$
be the corresponding right/down-path. With $(\ell,\ell)$ as
the unique grid point at which the path crosses the diagonal, as defined in Section~\ref{sec:prelim}, we note that $(B_1,\dots,B_{n-1})$
is a path from $(0,n-1)$ to $(\ell,\ell)$. We will attempt to change the parity by making a switch as close to the end $(\ell,\ell)$ of this path as possible. By choice of~$ \mathcal{F}$, there is a switch position at some
even $j\in [n-2k,n-2]$. We perform the switch at the last such position.
This only affects the last at most $2k$ positions of the area process,
which can only be increased by~1, reduced by 1 or stay the same. Therefore, the dominating
condition \eqref{eq:DC} is not affected (by our choice of $ \mathcal{F}$). This defines a matching
with the desired properties.
\end{proof}

\section{Computational results}\label{sec:comp}
In this section we give a new recursion which we have used to calculate $G(n)$ for many new values, and a description of how we numerically estimated the value of $\rho$ (and $\cdeg$). We also make the surprising observation that, while roughly half the sequences which satisfy the dominating condition have even sum, the convergence to a half is rather slow.

\subsection{Determining the exact values of \texorpdfstring{$G(n)$}{G(n)} for small \texorpdfstring{$n$}{n}}

\begin{figure}
    \centering
    \input{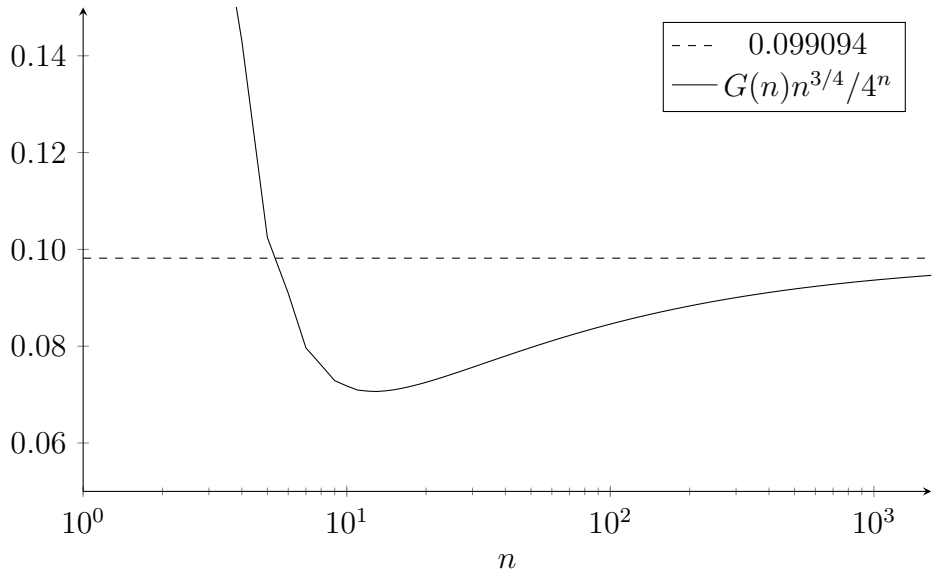}
    \caption{The graph depicts our numerical estimation of $\cdeg$ (dotted line) via the approximation of $\rho$, and a numerical estimation based on the exact values of $G(n)$ for small~$n$.}
    \label{fig:graph}
\end{figure}

We give a simple recursion to count the number of graphic sequences. 
In Section \ref{sec:reformulation}, steps 1 through 5 define a deterministic mapping from a sequence $n-1\geq d_1\geq \dots \geq d_n\geq 0$ to a (non-random) walk $Y$ (with steps in $\{-1,0,1\}$ and ending in $0$ or $-1$), such that the sequence $d$ satisfies the dominating condition (\ref{eq:DC}) if and only if $\sum_{j=1}^k Y_j\geq 0$ for all $1\leq k\leq n-1$. Moreover, as explained in Step 5, we have $\sum_{i=1}^n d_i\equiv \sum_{i=1}^{n-1}Y_i\bmod 2$.
We will use this to count the number of graphic sequences. 

Note that a single walk $Y$ corresponds to $2^z$ sequences, where $z$ is the number of zero steps that the walk takes (see Figure \ref{fig:reformulation}). We will weight each walk $Y$ accordingly when counting.

Let $F(N, y, a)$ be the weighted number of walks $(Y_i)_{i=1}^N$ which start at $y$, take $N$ steps, end in $\{-1,0\}$ and satisfy $a + \sum_{i=1}^k Y_i \geq 0$ for all $0 \leq k \leq N$ and $a + \sum_{i=1}^{N}Y_i \equiv 0 \bmod 2$. We think of $a$ as the starting value of the area process, or the area from any earlier steps of the walk. We see that by taking $N = n -1$, $y = 0$ and $a = 0$, we count exactly the walks which correspond to sequences that are graphic. Hence, the number of graphic sequences of length $n$ is $F(n-1, 0, 0)$, and we can calculate this using a recursion for $F$ as follows.

First, we have a boundary condition that $F(N, y, a) = 0$ whenever $a < 0$. If this is not the case, we consider the three options for the first step in the walk and how to complete the walk after this step. A completion would need to take $N - 1$ steps in every case, but the starting position and the area will change. The walk could step up, after which a completion starts at $y + 1$ with area $a + y + 1$; the walk could step down, in which case a completion starts at $y - 1$ with area $a + y  - 1$; or the walk could take a step of size $0$, in which case a completion starts at $y$ with area $a + y$.  Since in the last case the walk takes a step of size 0 and we are counting the weighted number of completions, this last case must be counted twice. This leads to the following recursion:
\[F(N, y, a) = F(N - 1, y + 1, a + y + 1) + F(N - 1, y - 1, a + y -1) + 2F(N-1, y, a + y),\]
with boundary condition $F(N, y, a) = 0$ whenever $a < 0$, and initial condition $F(0, y, a) = 1$ if $y \in \{0, -1\}$, $a \geq 0$ and $a$ even, and $F(0, y, a) = 0$ otherwise.

To calculate the number of graphic sequences, we calculated every value of $F$ for a given value of $N$ using the pre-computed values for $N - 1$, then calculated the values for $N  + 1$ using the values for $N$ and continued this until the calculation ran out of memory. At each step, we get one more value of $G(n)$ by reading off the value of $F(n-1, 0, 0)$. 
For this approach to work, we need $F(N, y, a)$ to take only finitely many values for each value of $N$. First, observe that $F(N, y, a) = 0$ whenever $y > N$ or $y < - N - 1$ as the walk needs to end in $\{0, -1\}$.

We now consider for which values of $a$ we need to calculate $F(N, y, a)$. Clearly, the boundary condition means we do not need to consider $a < 0$. Moreover, if $y < 0$, we can ignore all $a < y(y+1)/2$ as the area process will take a negative value at some point. This just leaves large $a$.

We claim that, given that the walk starts at $y$, in $N$ steps the area can decrease by at most 
\[
 a' = \frac{N^2 - 2Ny + 2N - y^2 + \one_{\text{odd}}(N-y)}{4}.
\]
In particular, this means that if $a \geq a'$, the area will always stay non-negative (provided it is already) and $F$ only depends on the parity of $a$. Hence, if $c = \max\{0, a'\}$ and $a \geq c$, the value $F(N, y , a)$ is either $F(N, y, c)$ or $F(N, y, c + 1)$. 

Let us briefly justify the claim. First, consider the amount the area process can decrease when the walk starts at 0 and returns to 0. The best thing for the walk to do is for it to take $\floor{N/2}$ steps downwards, then $\floor{N/2}$ steps upwards. If $N$ is odd, one extra step of size 0 should be inserted between the downwards and upwards steps. This reduces the area process by 
\[ \sum_{i=1}^{\floor{N/2}} i + \one_{\text{odd}}(N) \floor{N/2} + \sum_{i=1}^{\floor{N/2} - 1} i = \frac{N^2 - \one_{\text{odd}}(N)}{4}.\]
If $0 \leq y \leq N - 1$, the biggest reduction comes from taking $y + 1$ steps downwards to $-1$, then doing the above steps with the remaining steps (but starting and ending at $-1$). This reduces the area by 
\[ -\sum_{i = -1}^{y-1} i + \frac{(N - y - 1)^2 - \one_{\text{odd}}(N -y - 1)}{4} + (N - y - 1) = a'.\]
When $y = N$, the same equation holds. When $y < 0$, then the best thing to do is to reserve $-y - 1$ steps which will be used as upwards steps at the end. For the first $N + y + 1$ steps, one should aim to minimise the area while returning to $y$. This gives a total of 
\[ \frac{(N + y + 1)^2 - \one_{\text{odd}}(N + y + 1)}{4} - y (N + y + 1) + - \sum_{i= y + 1}^{-1} i = a'.\]
 The final optimisation we implement is to consider the function $f$ defined by 
\[f(N, y, a) = \begin{cases}
  F(N, y, a) - F(N, y, a - 1), & a \ne 0;\\
  F(N, y, 0), & a = 0.
 \end{cases}\]
This satisfies a similar recursion to the function $F$, but the numbers are generally smaller and this helps with the memory usage.

We wrote a program in Go which uses this recursion to calculate the number of graphic sequences and ran this on a node with 512 GiB of RAM until the program ran out of memory. This produced the first 1392 numbers (starting from $n = 0$). Of course, we are only interested in calculating $F(N, 0 ,0)$ and we do not need to calculate the value of $F(N -1, y, a)$ for every choice of $y$ and $a$. Indeed, calculating $F(N, 0, 0)$ only needs three values of $F(N, y, a)$ (and one is trivially 0). We therefore ran the program a second time and stopped the recursion just before the program ran out of memory. We then attempted to calculate $F(N + 1, 0, 0), F(N+2, 0, 0), \dots$ by only calculating values of $F$ when they were required (and keeping the calculated values in a map). 
This second step added another 260 values, so that we now know $G(n)$ for all $n \leq 1651$.
The code (along with the computed values) is attached to the arXiv submission, and the new values have been added to the OEIS entry \href{https://oeis.org/A004251}{A004251}. The sequence \href{https://oeis.org/A095268}{A095268} counts the number $G'(n)$ of zero-free graphic sequences, which is given by $G(n) - G(n - 1)$, and we have also updated this sequence.

\subsection{A surprising observation on the parity condition}
\label{subsec:eo}
By changing the initial condition, the above recursion can be changed to count $H(n)$, the sequences which satisfy the dominating condition and have odd parity.  Lemma \ref{lem:parity} shows that $H(n) =(1+o(1)) G(n)$, and one might naturally assume that $|G(n) - H(n)|$ is exponentially smaller than $G(n)$ (and $H(n)$), but this does not appear to be the case. 
Surprisingly, numerical estimates suggest that $G(n) - H(n) = \Theta(4^n/n^{5/2})$ which is only a factor of $\Theta(n^{7/4})$ smaller than $G(n)$. 

We remark that if we ignore the dominating condition and just count the number of sequences $n-1\ge d_1\ge\dotsb\ge d_n\ge0$ for which $\sum d_i$ is even and for which $\sum d_i$ is odd we get
\[
 \frac{1}{2}\left(\binom{2n-1}{n-1}+\binom{n-1}{\lfloor n/2\rfloor}\right)
 \qquad\text{and}\qquad
 \frac{1}{2}\left(\binom{2n-1}{n-1}-\binom{n-1}{\lfloor n/2\rfloor}\right)
\]
respectively. The difference between these two quantities is only $O(2^n)$, which is exponentially smaller than either. 

The partial matching we used to prove Lemma~\ref{lem:parity} suggests a possible explanation for this. In the proof, we switched at the last switchable position before $n-1$. Adapting the proof of the main result in \cite{AurzadaDereichLifshits2014} to \emph{lazy} SSRW bridges yields that $\Prb(\lwa_{n-1}=0\mid \lw_{n-1}=0,\lwa_1,\dots,\lwa_{n-1}\geq 0)=\Theta(n^{-7/4})$, and a sequence with $\lwa_{n-1} = 0$ cannot be `switched down'. Assuming that roughly half of them would be switched down (with the other half switched up), this yields a proportion $\Omega(n^{-7/4})$ of the even sequences which do not have an odd counterpart. 
Of course, there are also odd sequences without an even counterpart (e.g. those with $\lwa_{n-1} = 1$ and $\lwa_{n-2} = 0$ cannot be switched down), but this observation makes it at least reasonable that the order of the difference is $\Theta(4^n/n^{5/2})$.

\subsection{Estimating the constant \texorpdfstring{$\rho$}{rho}}
\label{subsec:rho}

We now give a method for estimating~$\rho$. Recall that $(\lw_k)$ is the lazy simple symmetric walk
and $(\lwac_k)$ is the random walk formed by summing the (signed) excursion areas $\lwae_k$ of $(\lw_k)$.
We first need to find the distribution of the excursion areas $\lwae_k$.
We define the following (partial) generating function
\[
 g(x,y)=\sum_{i,j>0}\Prb(\text{excursion is of area $i$ and length $j$})x^iy^j.
\]
Note that this is really only part of the full generating function, which is $\frac{y}{2}+g(x,y)+g(x^{-1},y)$, to take into account excursions below the axis, and also the probability $\frac{1}{2}$ event that the lazy
random walk does not move on the first step (and hence the excursion area is zero). 

\begin{lemma}
 The generating function $g(x,y)$ satisfies the recursive equation
 \[
  g(x,y)=\frac{xy^2}{16(1-\frac{xy}{2}-g(x,xy))}.
 \]
\end{lemma}
We note that as all terms of $g(x,y)$ have a positive power of~$y$, terms in $g(x,xy)$ have higher $x$-degree
than their counterparts in $g(x,y)$, and hence this equation allows one to recursively evaluate $g(x,y)$
to any order. For example,
\[
 g(x,y)=\tfrac{x y^2}{16}+
 \tfrac{x^2y^3}{32}+\tfrac{x^3y^4}{64}+
 \tfrac{x^4(y^4+2y^5)}{256}+
 \tfrac{x^5(y^5+y^6)}{256}+
 \tfrac{x^6(2y^5+3y^6+2y^7)}{1024}+O(x^7).
\]
\begin{proof}
To have a positive excursion, the first step of the lazy random walk must be up (probability $\frac{1}{4}$).
Then it either goes down (probability $\frac{1}{4}$) giving a term $\frac{1}{16}xy^2$ in $g(x,y)$,
or it follows some non-negative excursion before returning to height~1.
Let the excursion have area $X_1$ and length $T_1$ (so that the walk is now at $1$ with total area $X_1 + T_1 + 1$ and length $T_1 + 1$).
The rest of the excursion is then
equivalent to one starting at $0$ at time $T_1$ with an initial step up, and say this excursion has area $X_2$ and length $T_2$.
The sum $\sum \Prb(X_1,T_1)x^{X_1}z^{T_1}$ is just $\frac{z}{2}+g(x,z)$ (as we allow the
trivial area zero excursion here), and the sum $\sum \Prb(X_2,T_2)x^{X_2}y^{T_2}$is just $4g(x,y)$
(as we automatically can assume the first step is up). Overall we get an
excursion of length $T_1+T_2$ and area $(X_1+T_1)+X_2$. Setting $z=xy$ and multiplying these (together with 
a factor of $\frac{1}{4}$ for the initial step up) gives the remaining terms of $g(x,y)$. Hence,
\[
 g(x,y)=\tfrac{xy^2}{16}+\tfrac{1}{4}\big(\tfrac{xy}{2}+g(x,xy)\big)\cdot 4g(x,y),
\]
and the lemma follows from rearranging this equation.
\end{proof}

Now by setting $y=1$ we get a generating function for
the (positive) excursion areas, namely
\[
 g(x,1)=\tfrac{x}{16}+\tfrac{x^2}{32}+\tfrac{x^3}{64}+
 \tfrac{3x^4}{256}+\tfrac{x^5}{128}+\tfrac{7x^6}{1024}+
 \tfrac{21x^7}{4096}+\tfrac{37x^8}{8192}+\tfrac{31x^9}{8192}+O(x^{10}).
\]
Note that $g(1,1)=\frac{1}{4}$ is the probability of a positive excursion.

The next task is to estimate or bound~$\rho$, which is the probability that the random walk $\lwac_k$ with step
sizes $\lwae_k$ hits 0 before first taking a negative value. To do this we construct a finite state Markov chain
with states $\{-,0,1,2,\dots,n-1,\star\}$ where the states $-$ and $\star$ represent the walk taking negative value and reaching a state at least $n$ respectively. The states $-$ and $\star$ are absorbing, and otherwise we add a random variable distributed like a signed excursion area. If adding this, takes the walk negative or at least $n$, the walk moves to $-$ or $\star$  as necessary. We start the Markov chain at 0 and run until we either hit 0 again, or one of the states $-$ or $\star$. 
By a simple coupling argument it is clear that
\[
 \Prb(\text{hit }0)\le \rho \le \Prb(\text{hit }0) + \Prb(\text{hit }\star).
\]
Taking $n$ sufficiently large gives us reasonable bounds on~$\rho$.

As an example, taking $n=2$ we have states $\{-,0,1,\star\}$ and transition matrix
\[\begin{pmatrix}
1&0&0&0\\
\frac14&\frac12&\frac{1}{16}&\frac{3}{16}\\
\frac{3}{16}&\frac{1}{16}&\frac12&\frac14\\
0&0&0&1
\end{pmatrix}.\]
Writing $h_{ij}$ for the probability of hitting $j$ starting at $i$, we have
\[
 h_{10}=\tfrac{1}{16}+\tfrac{1}{2}h_{10},\qquad
 h_{1\star} = \tfrac{1}{4}+\tfrac{1}{2}h_{1\star},
\]
giving $h_{10}=\tfrac{1}{8}$, $h_{1\star}=\tfrac{1}{2}$. Then
\[
 h_{00}=\tfrac{1}{2}+\tfrac{1}{16}h_{10}=\tfrac{65}{128},\qquad
 h_{0\star}=\tfrac{3}{16}+\tfrac{1}{16}h_{1\star}=\tfrac{7}{32}.
\]
Hence,
\[
 \tfrac{65}{128}\le \rho \le \tfrac{65}{128}+\tfrac{7}{32}=\tfrac{93}{128}.
\]
Using this method with $n=2^{18}$ gives a lower bound of $0.51580258$, which appears to be very close to
the true value of~$\rho$. However the upper bound of $0.54543568$ obtained by this method
seems to still be very far from the truth.

If we make the assumption
that hitting 0 before first taking a negative value is a decreasing function of the starting point, we
can amalgamate the states $n-1$ and $\star$ in the above model, giving transitions out of
that state as if they were from $n-1$. Unfortunately we do not have a proof that $h_{i0}$ is decreasing in~$i$,
so this does not give a rigorous bound on~$\rho$. Nevertheless, for $n=2^{18}$ we obtain a (non-rigorous)
upper bound of $0.51580289$ by this method, which does seem much closer to the true value.
An even less rigorous estimate can be obtained by applying Richardson extrapolation
to $h_{00}$ in terms of $1/n$, which gives the estimate 
\[
 \rho\approx 0.515802638089141858504490255841,
\]
and corresponds to a value of
\[
 \cdeg\approx 0.099094083237488745361449340935.
\]
These approximations do not appear to correspond to any number with a simple closed form expression.

\section{Concluding remarks}
\label{sec:conclusion}

We have given a precise asymptotic for the number of graphical sequences.
Similar asymptotics are known for tournaments, but not for other natural classes such as uniform hypergraphs and digraphs. 
Another interesting direction with open problems is the asymptotics of the number of graphical partitions of an integer $N$ (that is, the number of graphic sequences which sum to $N$). 

We outline some more directions for future research below.
\paragraph{An upper bound in Lemmas \ref{lem:useful} and \ref{lem:useful2}}
For $x\in \R_{>0}^n$, let 
\[\textstyle
  A(x)=\big\{(\sigma,s):\sigma \in S_n,\,s\in \{-1,1\}^n,\,
  \sum_{i=1}^k s_ix_{\sigma(i)}\ge 0 \text{ for all }k\in [n]\big\}.
\]
Lemma \ref{lem:useful} states that $|A(x)|\ge (2n-1)!!$, with equality for all $x$ for which all sums $\sum_{i\in S}x_i$ are distinct (for distinct $S\subseteq[n]$).
\begin{conjecture}
 \label{conj:comb}
 For $x\in \R_{>0}^n$,
 $|A(x)|$ is maximized when $x_1=\dots=x_n$.
\end{conjecture}
Rephrased in probabilistic terms, we conjecture that the simple symmetric random walk has the highest probability of staying non-negative, amongst all random processes with exchangeable increments $(X_1,\dots,X_n)\in (\R\backslash\{0\})^n$ of which the law is invariant under sign changes of the elements. 

This conjecture would imply (see \eqref{e:ulbounds}) that for all 
 exchangeable $(X_1,\dots,X_n)\in (\R\backslash\{0\})^n$ of which the law is invariant under sign changes of the elements, 
  \[ \frac{1}{\sqrt{\pi (n+1/2)}}\le \Prb\left(\sum_{i=1}^k X_i \ge 0\text{ for all }k\in [n]\right)\le \frac{\sqrt{2}}{\sqrt{\pi n}}.\]
 In particular, for this very general class of random processes (that contains all symmetric random walks), the order of the probability of staying non-negative does not depend on the law, or even the tail behaviour of the increments.
 
\paragraph{Uniformly random graphic sequences}
A natural next question, that we intend to answer in future work, is the (asymptotic) law of a uniformly random graphic sequence of length $n$. We now make some observations using our reformulation and discuss a potential strategy to answer this question. Firstly, note that a uniform lattice path has the law of a simple symmetric random walk bridge with $2n$ steps that straddles the line $y=n-x$, and in particular, for large $n$, its fluctuations away from the line $y=n-x$ are of order $n^{1/2}$. In fact, for any $\delta>0$, it is exponentially unlikely that at some point the lattice path is at distance more than $\delta n$ from the line  $y=n-x$. This means that even after conditioning on the lattice path encoding a graphic sequence (i.e.\ conditioning on an event with probability $\Theta(n^{-1/4})$), it is exponentially unlikely that at some point the lattice path is at distance more than $\delta n$ from the line  $y=n-x$, and in general, the large deviations of a uniformly random graphic sequence are completely described by the large deviations of a uniformly random lattice path. (Observe that this implies that a uniform graphic sequence is very different from the degree sequence of a uniform graph, of which the corresponding lattice path stays close to the horizontal line $y=n/2$.)

Therefore, to observe the difference between a uniform graphic sequence and a uniform lattice path, we will need to consider their more fine-grained behaviour, for example by studying the scaling limit of their fluctuations around the line $y=n-x$. We conjecture the following.
\begin{conjecture}
 Let $D_1\ge\dotsb\ge D_n$ be a uniformly random graphic sequence of length $n$. Then, there exists a random continuous function $D$ from $[0,1]$ to $\R$ such that 
 \[\left(n^{-1/2}(D_{\lfloor tn\rfloor}-(1-t)n),0\le t\le 1\right)\overset{d}{\to}\left(D_t,0\le t\le 1\right)\]
 in the uniform topology.
\end{conjecture}
We expect $D$ to have two characterizations. Firstly, it can be defined as a Brownian bridge conditioned to satisfy a continuous version of the dominating condition \eqref{eq:DC}.  Secondly, via a continuous version of the reformulation described in Section \ref{sec:reformulation}, $D$ can be constructed via two paths, for which the distance between them is distributed as a conditioned Brownian bridge and their midpoint is determined by a Brownian motion (the Brownian motion plays the role of the `lazy steps' that can go either right or down). This result would add graphic sequences to the long and varied list of uniformly random combinatorial structures with a `Brownian' scaling limit; examples are numerous models of trees of which the scaling limit can be described by a Brownian excursion (see the survey paper \cite{LeGall}),  mappings with a limit described by the Brownian bridge \cite{randommappings}, various classes of maps with limits encoded by the Brownian snake (see the survey paper \cite{surveyRandomMaps}) and pattern-avoiding permutations with limits described by a Brownian excursion (see \cite{randompermutations} and references therein).   

Such scaling limits are interesting in their own right, but can also be exploited to answer questions about the corresponding combinatorial class, such as `what proportion of Cayley trees of size $n$ have height exceeding $tn^{1/2}$?' or `in what proportion of maps from $[n]$ to $[n]$ is the average distance to a cycle larger than $tn^{1/2}$?'.

\paragraph{Persistence probabilities of integrated random processes}
The study of the probability that integrated random walks and random walk bridges stay non-negative started with the work of Sina\u{\i} on the SSRW \cite{Sinai1992} and has attracted a lot of attention in the past decade (\cite{AurzadaDereichLifshits2014,Dembo2013,Denisov2015, Gao2014,Vysotsky2010,Vysotsky2014}; see also the survey paper \cite{Aurzada2015}).  However, all work on integrated random walk bridges only finds the right order of the persistence probability \cite{AurzadaDereichLifshits2014,Vysotsky2014} and for random walks the sharp asymptotics (including the value of the constant) are only known under a $(2+\delta)$-moment condition on the step distribution. 
Our methods completely carry over to the setting of SSRW bridges, yielding the following result of independent interest. Let $\uw$ be a SSRW and let $\uwa$ be its area process.
\begin{prop}\label{prop:persistenceSSRW}
 We have that 
  \[
  n^{1/4}\Prb(\uwa_1,\dots,\uwa_{2n}\ge 0\mid \uw_{2n}=0) \to \frac{\Gamma(3/4)}{\sqrt{2\pi(1-\hat{\rho})}}
 \]
 as $n\to \infty$, for $\hat{\rho}$ the probability that the random walk with steps distributed as the signed area of the excursions of a SSRW hits $0$ before first taking a negative value.
\end{prop}
Adapting the method to numerically estimate the value of $\rho$ described in Section~\ref{subsec:rho} shows that $\hat{\rho}\approx 0.0773408571485249705089600725$.

This approximation can be confirmed with a more direct expression for $\hat\rho$ that appears in a follow-up work by the second author and Kolesnik \cite{Kolesnik24}. There it is shown that 
$\hat\rho = 1-e^{-\hat\xi}$ where \[\hat\xi=\sum_{n=1}^\infty \frac{1}{4n^2 2^{4n}} \sum_{d\mid 2n} \binom{2d-1}{d}\phi(2n/d)\]
with $\phi$ Euler's totient function.

It is possible that our techniques can be generalized to other models of random bridges. However, there are some difficulties to overcome, both conceptual and technical. Our proof relies heavily on Lemma~\ref{lem:useful2}, which requires the areas of different excursions of the bridge to be exchangeable and for their law to be invariant under sign changes. When only considering random walk bridges, these conditions combined restrict the method to (scalings of) symmetric processes with steps in $\{-1,0,1\}$, 
so new ideas are needed to adapt the method to other random walk bridges. 
 
\paragraph{Acknowledgements} The exact values of $G(n)$ were computed using the computational facilities of the \href{http://www.bris.ac.uk/acrc/}{Advanced Computing Research Centre, University of Bristol}. We would like to thank Peter Winkler for pointing out the upper bound in Lemma~\ref{lem:useful}, and the anonymous referee for their helpful comments.
 
\bibliography{bibliography.bib}

\appendix
\newpage

\section{Overview of notation}
\label{sec:notation_overview}
The following notation is used in the paper.
\begin{itemize}
    \item $G(n)$ = the number of graphic sequences of length $n$.
    \item $\lw=(\lw_k)_{k\geq 0}$=lazy random walk, that is, a random process with independent and identically distributed increments that take the value $0$ with probability $1/2$, $-1$ with probability $1/4$ and $+1$ with probability $1/4$.
    \item $\lwa_k=\sum_{i=0}^k \lw_i$, so $(\lwa_k)_{k\geq 0}$ is an integrated random walk.
    \item $\lwae_i$ = area of the $i$th excursion of $\lw$.
    \item $\lwac_k=\sum_{i=1}^k \lwae_i$. We can see $\lwac$ as a `subsequence' of $\lwa$, where $\lwa\ge 0$ if and only if $\lwac\ge 0$.
    \item $\zeta_1=\inf\{k\geq1 : \lw_k=0, \lwa_k\le 0\}$.
    \item $\rho = \Prb(\lwa_{\zeta_1}=0)=$ probability that $(\lwa_k)_{k\geq 0}$ hits zero before the first time it goes negative.
    \item  $\lb=(\lb_k)_{k=0}^n$= lazy random walk bridge, that is, $(\lw_k)_{k=0}^n$ conditioned on $\lb_n=0$.
    \item $\lba,\lbae,\lbac$ the counterparts of $\lwa,\lwae,\lwac$ respectively, with $\lw$ replaced by $\lb$ in the definition.
    \item $N_n$ = number of times $\lb$ hits 0 after time 0 (so $N_n\in \{1,\dots,n\}$).
    \item $M_n$ = the number of times $\lbac$ hits zero.
    \item $\pbac,\pbae$ = perturbed variants of $\lbac$ and $\lbae$ respectively.
    \item $\xi_i$ = the $i$th time $\lbac$ hits zero.
    \item $\probareazero{n}=\Prb(M_n\ge 1\mid \lbac_1,\dots, \lbac_{N_n}\ge 0)$.
\end{itemize}

\section{Proof of Lemma \ref{lem:useful}}

\useful*
We include the proof of Lemma \ref{lem:useful} from Burns \cite{Burns2007TheNO} below for the convenience of the reader.

For $x\in \R_{>0}^n$, let 
\[\textstyle
  A(x)=\big\{(\sigma,s):\sigma \in S_n,\,s\in \{-1,1\}^n,\,
  \sum_{i=1}^k s_ix_{\sigma(i)}\ge 0 \text{ for all }k\in [n]\big\}.
\]
We will show that
\[
 |A(x)|\ge (2n-1)!!
\]
and that equality holds if, for all distinct $S,S'\subseteq [n]$, the
corresponding sums are also distinct.
 
We call a vector $y\in \R_{>0}^n$ \emph{rapidly decreasing} if
\[
 y_i > y_{i+1}+\dotsb+y_n
\]
for all $i\in [n]$. We will first use induction on $n$ to show that $|A(y)|=(2n-1)!!$ for all rapidly decreasing
$y\in \R_{>0}^n$. It is clear that $A(y)=\{(\text{Id},1)\}$ when $n = 1$, and the claim holds.

Suppose that we have shown the claim for some $n\ge 1$, and let $y\in \R_{>0}^{n+1}$. Since the sequence is
rapidly decreasing, the pair $(\sigma, s)$ is in $A(y)$ if and only if $s(i) = 1$ for all $i \in [n+1]$
such that $\sigma(i)$ is the lowest number seen so far, i.e. $\sigma(i) < \sigma(j)$ for all $j < i$.
Let $\alpha = \sigma^{-1}(n+1)$. If $\alpha \ne 1$, then the term $s_\alpha y_{n+1}$ has no impact on whether
the sequence is valid, while if $\alpha = 1$, we require $s_1=1$ and then we need the remainder of the
sequence to satisfy the condition. More formally, let $y'$, $\sigma'$ and $s'$ be what is left after
removing $y_{n+1}$, $n+1$ and $s_{\alpha}$ respectively. That is 
\[
 y'= (y_1,\dots,y_n),\quad \sigma'(i) = \sigma\big(i+\one_{\ge\alpha}(i)\big),
 \quad s' = (s_1,\dots,s_{\alpha-1},s_{\alpha+1},\dots,s_{n+1}).
\]
If $\alpha = 1$, then $(\sigma,s)\in A(y)$ if and only if $s_1=1$ and $(\sigma',s') \in A(y')$.
If $\alpha\ne 1$, then $(\sigma,s)\in A(y)$ if and only if $(\sigma',s') \in A(y')$.
Hence, there are $2n+1$ pairs $(\sigma,s)$ in $A(y)$ for each pair $(\sigma',s')$ in $A(y')$.
Clearly, $y'$ is also rapidly decreasing and the result follows by induction. 

For $x\in \R^n$ and $S\subseteq [n]$, we write $x_S=\sum_{i\in S}x_i$. We call $x\in \R_{>0}^n$
\emph{sum-distinct} if $x_S\ne x_{S'}$ for all distinct subsets $S,S'\subseteq[n]$.
Let $x\in \R^n_{>0}$ be sum-distinct. We will construct $y\in \R_{>0}^n$ that is rapidly decreasing and for which $|A(x)|=|A(y)|$.

Since $x$ is sum-distinct, we can define a total order $<_x$ on the power set $\mathcal{P}([n])$ by $S <_x S'$ if and only if
$x_S < x_{S'}$. We note that for a given $\sigma$ and~$s$, the condition $\sum_{i=1}^k s_ix_{\sigma(i)}\ge 0$
is equivalent to the condition $x_{S_-} <_x x_{S_+}$ where $S_\pm = \{\sigma(i): i \le k,\,s_i = \pm1\}$,
and hence $A(x)$ only depends on $x$ through $<_x$. 

To get to the vector $y$, we will increase $x_1$ until it is larger than $x_2+\dotsb+x_n$, and then increase
$x_1$ and $x_2$ until we also have $x_2>x_3+\dotsb+x_n$ etc. We will do this in a series of steps so that each
increase only changes the total ordering by a ``small" amount, and the following claim shows this
preserves $|A(x)|$. We defer the proof of this claim to the end of this section and first finish the current proof.

\begin{claim}\label{claim:small-change}
 Let $x$ and $x'$ be sum-distinct and assume there is a unique pair $\{L,R\}$ of disjoint subsets for which
 $L<_{x}R$ yet $R<_{x'}L$. Then $|A(x)| = |A(x')|$.
\end{claim}

Suppose that we already have $x_{j} > x_{j+1}+\dotsb+x_{n}$ for all $j<i$, and we wish to extend this to include
$j=i$ as well. We will slowly increase $x_1$, $x_2$, \dots, $x_i$ so that we only change one disjoint inequality
on the power set at a time and we can use the claim above to show that $|A(x)|$ does not change. We first ensure
that no signed sums are the same by slightly perturbing $x$ by a small amount, which we choose to be small enough to not
change the order $<_x$. Indeed, if
\[
 \eps=\min_{S\ne S'}|x_S-x_{S'}|=\min_{S\ne S'}|x_{S\setminus S'}-x_{S'\setminus S}|,
\]
then perturbing each entry of $x$ by less than $\eps/n$ cannot possibly change the order $<_x$.
Let $z$ be a random vector formed by adding a small independent $\Unif[0,\eps/n]$ random variable to each entry
of~$x$, and note that $\mathord{<_x}=\mathord{<_z}$,
so $|A(x)| = |A(z)|$. Almost surely there are no two pairs of disjoint sets
$(A,B)$ and $(C,D)$ such that $z_A-z_B = z_C-z_D$, and we can order the pairs $(A_t,B_t)$ of disjoint subsets
of $[n] \setminus [i]$ such that $z_{A_t}-z_{B_t}$ is increasing in~$t$. Suppose $(A_\tau,B_\tau)$ is the
first pair for which $z_{A_\tau}-z_{B_\tau} > z_i$. If there is no such pair, then we already have
$z_i > z_{i+1}+\dotsb+ z_{n}$ (by choosing $A=[n]\setminus [i]$ and $B=\emptyset$).
Let $\delta$ be any value in the interval $(z_{A_\tau}-z_{B_\tau}-z_i, z_{A_{\tau+1}}-z_{B_{\tau+1}}-z_i)$
(or in $(z_{A_\tau}-z_{B_\tau}-z_i,\infty)$ if there is no pair $(A_{\tau+1},B_{\tau+1}))$, and consider the vector 
\[
 z^{(2)} =(z_1+2^{i-1}\delta, z_2+2^{i-2}\delta, \dots, z_i+\delta, z_{i+1}, z_{i+2}, \dots z_n).
\]
This is again sum-distinct and there is a unique pair $(L,R)=(B_\tau\cup\{i\},A_\tau)$ such that
$L <_z R$ but $R <_{z^{(2)}} L$, so $|A(z^{(2)})| = |A(z)| = |A(x)|$. It also still has the property
that  $z^{(2)}_{j} > z^{(2)}_{j+1}+\dotsb+z^{(2)}_{n}$ for $j<i$. It may not yet have the property
that $z_{i}^{(2)} > z_{i+1}^{(2)}+\dotsb+z_{n}^{(2)}$, but we do have
$z^{(2)}_i>z_{A_\tau}-z_{B_\tau}$, and we can choose $\delta^{(2)}$ in the
interval $(z_{A_{\tau+1}}-z_{B_{\tau+1}}-z^{(2)}_i, z_{A_{\tau+2}}-z_{B_{\tau+2}}-z^{(2)}_i)$ 
and repeat to get $z^{(3)}$, $z^{(4)}$, $\dots$. The process terminates at
some $k$ when $z^{(k)}_i > z^{(k)}_{i+1}+\dotsb+z^{(k)}_{n}$, and we take this to be our new~$x$. 
Repeating this process for $i = 1,2,\dots,n-1$ in turn, gives a rapidly decreasing vector
$y$ with $|A(y)| = |A(x)|$, as required.

We still need to prove the claim for $x$ which are not sum-distinct. As before let 
\[
 \eps=\min_{x_S\ne x_{S'}}|x_S-x_{S'}|=\min_{x_S\ne x_{S'}}|x_{S\setminus S'}-x_{S'\setminus S}|.
\]
If each entry is perturbed by less than $\eps/n$ to get a vector $z$, then 
\[
 \sum_{i=1}^k s_ix_{\sigma(i)}\ge 0 \iff \sum_{i=1}^k s_iz_{\sigma(i)} > -\eps
 \impliedby \sum_{i=1}^k s_iz_{\sigma(i)} \ge 0.
\]
Therefore, $A(x) \supseteq A(z)$. If we get $z$ by by adding a small $\Unif[0,\eps/n]$
random variable to each entry of~$x$, then $z$ is almost surely sum-distinct and the
result follows since $|A(x)| \ge |A(z)| = (2n - 1)!!$ almost surely.

We now return to the proof of Claim~\ref{claim:small-change}.
\begin{proof}[Proof of Claim \ref{claim:small-change}]
Let $k = |L\cup R|$, and consider the function $f\colon S_n\times\{-1,1\}^n \to S_n\times\{-1,1\}^n$
which acts as follows. For any $(\sigma,s) \in S_n\times\{-1,1\}^n$, the function $f$ maps
$(\sigma,s)$ to $(\sigma',s')$ where $\sigma'$ is the permutation which is reversed on the first $k$ inputs,
and $s'$ is formed by reversing the order of the first $k$ entries in $s$ and negating them. 
That is, $\sigma'(i)=\sigma(k+1-i)$ for $i\in [k]$ and $\sigma'(i)=\sigma(i)$ for $i\in [k+1,n]$,
and $s_i'=-s_{k+1-i}$ for $i\in [k]$ and $s_i'=s_i$ for $i\in [k+1,n]$.

The function $f$ is clearly self-inverse and hence bijective, and we will show that
$f(A(x)\setminus A(x')) \subseteq A(x')\setminus A(x)$. By switching the roles of $x$ and $x'$ and of $L$ and $R$, it follows that $f(A(x')\setminus A(x))\subseteq A(x)\setminus A(x')$, and so $|A(x)| = |A(x')|$.

Suppose that $(\sigma,s) \in A(x)\setminus A(x')$ and let $f(\sigma,s) = (\sigma',s')$.
It is obvious that $(\sigma',s')\notin A(x)$ as 
\[
\sum_{i=1}^k s_i'x_{\sigma'(i)}=-\sum_{i=1}^k s_ix_{\sigma(i)}<0.
\]

Since $(\sigma,s) \notin A(x')$, there must be at least one $\ell$ for which $\sum_{i=1}^\ell s_ix_{\sigma(i)}'<0$.
We first show that there is exactly one choice for~$\ell$, and that it is~$k$.
Let $S_{\pm}=\{\sigma(i) : i\in [\ell],\,s_i=\pm 1\}$. Then
\begin{align*}
 \sum_{i=1}^\ell s_ix_{\sigma(i)} = x_{S_+}-x_{S_-} > 0,\\
 \sum_{i=1}^\ell s_ix'_{\sigma(i)} = x'_{S_+}-x'_{S_-} < 0.
\end{align*}
In other words, $S_- <_x S_+$ and $S_+ <_{x'} S_-$. Since $S_+\cap S_-=\emptyset$,
we find that $(S_-,S_+)=(L,R)$ and the only option for $\ell$ is~$k$.

Using this we can check that $(\sigma',s') \in A(x')$. For $j\in[k]$, we have
\begin{align*}
 \sum_{i=1}^j s'_i x'_{\sigma'(i)}
 &=-\sum_{i=k-j+1}^{k} s_i x'_{\sigma(i)}\\
 &=-\sum_{i=1}^k s_i x'_{\sigma(i)}+\sum_{i=1}^{k-j} s_ix'_{\sigma(i)}.
\end{align*}
Both of these terms are non-negative since $\sum_{i=1}^{\ell} s_ix'_{\sigma(i)}< 0$ if and only
if $\ell=k$. Similarly, for $j>k$, we have
\begin{align*}
 \sum_{i=1}^j s'_i x'_{\sigma'(i)}
 &=\sum_{i=1}^k s'_i x'_{\sigma'(i)}-\sum_{i=1}^k s_i x'_{\sigma(i)}+\sum_{i=1}^j s_i x'_{\sigma(i)}\\
 &=-2\sum_{i=1}^k s_ix'_{\sigma(i)}+\sum_{i=1}^js_ix'_{\sigma(i)}\ge 0.
\end{align*}
Hence, $(\sigma',s') \in A(x')\setminus A(x)$.
\end{proof}

\section{Proof of Lemma~\ref{lem:locallimitlazy}}\label{app:locallimitlazy}
\locallimitlazy*
The proof is an adaptation of the proof of Proposition~1 in~\cite{AurzadaDereichLifshits2014}.

Let $a,b\in\Z$. We will use Fourier inversion to estimate $\Prb(\lw_n=a,\lwa_n=b)$. We let $\lmgf_n\colon\R^2\to \mathbb{C}$ given by
\[
 (t_1,t_2)\mapsto \E\left[e^{i(t_1\lw_n+t_2\lwa_n)}\right]
\]
be the characteristic function of $(\lw_n,\lwa_n)$, so that by $2$-dimensional Fourier inversion, we have
\begin{equation}
 \Prb\left(\lw_n=a,\,\lwa_n=b\right)
 =\frac{1}{(2\pi)^2}\int_{-\pi}^\pi\!\int_{-\pi}^\pi \lmgf_{n}(t_1,t_2)e^{-i(t_1a+t_2b)}dt_1dt_2.
 \label{eq:fourierinversion}
\end{equation}
We observe that, for $\lws_1,\lws_2,\dots$ i.i.d.\ random variables distributed as the steps
of $\lw$ (i.e. $\Prb(\lws_i=0)=\tfrac{1}{2}$, $\Prb(\lws_i=-1)=\Prb(\lws_i=1)=\tfrac{1}{4}$) we have
\[
 \lws_1\overset{d}{=}\tfrac{1}{2}(\uws_1+\uws_2)
\]
for $\uws_1$ and $\uws_2$ two i.i.d.\ Bernoulli random variables with
$\Prb(\uws_1=1)=\Prb(\uws_2=-1)=\tfrac{1}{2}$. Therefore, the characteristic function of $\lws_1$ satisfies 
\[
 \E\left[e^{it\lws_1}\right]
 =\E\left[e^{it\uws_1/2}\right]^2
 =\cos^2\left(\tfrac{t}{2}\right).
\]
Moreover, note that
\[\textstyle
 (\lw_n,\lwa_n)\overset{d}{=}\left(\sum_{k=1}^n \lws_k,\sum_{k=1}^n (n-k+1)\lws_k\right)
 \overset{d}{=} \left(\sum_{k=1}^n \lws_k,\sum_{k=1}^n k\lws_k\right),
\]
so 
\[
 \lmgf_n(t_1,t_2)=\prod_{k=1}^n\E\left[e^{i(t_1+kt_2)\lws_k}\right]
 =\prod_{k=1}^n\cos^2\left(\tfrac{t_1+kt_2}{2}\right).
\]

In particular, we observe that for $n\ge 2$, the absolute value of the integrand in \eqref{eq:fourierinversion}
is equal to $1$ if and only if $t=(t_1,t_2)=(0,0)$ and is strictly smaller otherwise. We will examine
the contribution to the integral of $t$ in a small region around $(0,0)$, and we will show
that the contribution is negligible outside of that region. 

Define $T_1=\{(t_1,t_2)\in[-\pi,\pi]^2:|t_1|+n|t_2|\ge\pi\}$. Then, for $n\ge 2$, it
is not too hard to see that there exists a $c<1$ such that $|\cos(\frac{t_1+kt_2}{2})|\le c$
for at least half of the values of $k=1,\dots,n$. Hence, on $T_1$,
\[
 |\lmgf_n(t_1,t_2)|\le c^n=\exp(-\Omega(n)).
\]
Now define $T_2=\{(t_1,t_2):n^{-1/3}\le |t_1|+n|t_2|<\pi\}$. 
We observe that we have the bound $\cos^2x\le e^{-x^2}$ for $|x|\le\frac{\pi}{2}$ and
\begin{equation}\label{e:sumsq}
 \sum_{k=1}^n(t_1+kt_2)^2=nt_1^2+n(n+1)t_1t_2+\tfrac16n(n+1)(2n+1)t_2^2.
\end{equation}
Hence, on $T_2$, we have
$\sum_{k=1}^n(t_1+kt_2)^2=n(t_1+(n+1)t_2/2)^2+n(n^2-1)t_2^2/12=\Omega(n^{1/3})$ as
$\max\{|nt_2|,|t_1+(n+1)t_2/2|\}=\Omega(n^{-1/3})$. Thus
\[
 |\lmgf_n(t_1,t_2)|\le\exp\Big(-\tfrac14\sum_{k=1}^n(t_1+kt_2)^2\Big)=\exp(-\Omega(n^{1/3})).
\]
We deduce that
\[
 \frac{n^2}{(2\pi)^2}\iint_{T_1\cup T_2}\lmgf(t_1,t_2)e^{-i(t_1a+t_2b)}dt_1dt_2=o(1)
\]
as $n\to\infty$, uniformly in $a$ and $b$.

Finally consider $T_3=\{t:|t_1|+n|t_2|<n^{-1/3}\}$.
Now for small $x$, $\cos^2x=\exp(-x^2+O(x^4))$, and so \eqref{e:sumsq} implies that on $T_3$ we have
\[
 \lmgf_n(t_1,t_2)=\exp\left(-\tfrac{1}{4}\left(nt_1^2+n^2t_1t_2+\tfrac{n^3}{3}t_2^2+O(n^{-1/3})\right)\right).
\]
Writing $t_1=n^{-1/2}s_1$, $t_2=n^{-3/2}s_2$, $v_1=n^{-1/2}a$ and $v_2=n^{-3/2}b$, we have 
\begin{align*}
 \frac{n^2}{(2\pi)^2}&\iint_{T_3}\lmgf(t_1,t_2)e^{-i(t_1a+t_2b)}dt_1dt_2\\
 &=\frac{n^2}{(2\pi)^2}\iint_{T_3}
  \exp\left(-\tfrac{1}{4}\left(nt_1^2+n^2t_1t_2+\tfrac{n^3}{3}t_2^2+O(n^{-1/3})\right)\right)e^{-i(t_1a+t_2b)}dt_1dt_2\\
 &=\frac{1}{(2\pi)^2}\iint_{|s_1|+|s_2|\le n^{1/6}}\!\!
  \exp\left(-\tfrac{1}{4}\left(s_1^2+s_1s_2+\tfrac{1}{3}s_2^2+O(n^{-1/3})\right)\right)e^{-i(s_1v_1+s_2v_2)}ds_1ds_2\\
 &\to \frac{1}{(2\pi)^2}\iint_{\R^2}
  \exp\left(-\tfrac{1}{4}\left(s_1^2+s_1s_2+\tfrac{1}{3}s_2^2\right)\right)e^{-i(s_1v_1+s_2v_2)}ds_1ds_2
\end{align*}
as $n\to\infty$ uniformly over all $v_1$ and $v_2$. Now
\begin{multline*}
 \frac{1}{(2\pi)^2}\iint\exp\left(-\tfrac{1}{4}\left(s_1^2+s_1s_2+\tfrac{1}{3}s_2^2\right)\right)
  e^{-i(s_1v_1+s_2v_2)}ds_1ds_2\\
 =\frac{1}{(2\pi)^2}\iint\exp\left(-\tfrac{1}{2}s^TR^{-1}s\right)e^{-i s^Tv}ds_1ds_2
 =\frac{\sqrt{\det R}}{2\pi}\exp\left(-\tfrac{1}{2}v^TRv\right),
\end{multline*}   
where 
\[
 s=\begin{pmatrix}s_1\\s_2\end{pmatrix},\quad
 v=\begin{pmatrix}v_1\\v_2\end{pmatrix},\quad
 R^{-1}=\frac{1}{2}\begin{pmatrix}1&\tfrac{1}{2}\\\tfrac{1}{2}&\tfrac{1}{3}\end{pmatrix},\quad
 R=\begin{pmatrix}8&-12\\-12&24\end{pmatrix}, \quad \det R=48.
\]
Therefore, we deduce that
\[
 n^2\Prb\left(\lw_n=a,\lwa_n=b\right)
 \to\frac{\sqrt{\det R}}{2\pi}\exp\left(-\tfrac{1}{2}v^TRv)\right)\\
 =\frac{2\sqrt{3}}{\pi}\exp\left(-4v_1^2+12v_1v_2-12v_2^2\right)
\]
as $n\to\infty$, uniformly over all $(v_1,v_2)$, as required. 

\section{Proof of Lemma \ref{lem:upperboundintegral0}}\label{app:upperboundintegral}
\upperboundintegral*
With Lemma~\ref{lem:locallimitlazy} in hand, the proof of Lemma \ref{lem:upperboundintegral0} is a direct adaptation of the proof of the
upper bound of Theorem~1 of \cite{AurzadaDereichLifshits2014} for simple symmetric random walks.
We include the proof for our case for completeness. Let $\adjlw=\adjlw(\lw,n)$ be the
adjoint process of $\lw$ on $[n]$, i.e. for $i\in [n]$, we set $\adjlw_i:=\lw_n-\lw_{n-i}$
and let $\adjlwa$ be the area process of~$\adjlw$. 

Denote $\lw_k=\sum_{i=1}^k\lws_i$ for all $k$, set $b=\lfloor n/4 \rfloor$, and define the events
\begin{align*}
 \Omega^+_n&=\{\lwa_1,\dots,\lwa_b\ge 0 \}\in \sigma(\lws_1,\dots,\lws_b), \\
 \bar{\Omega}^+_n&=\{\adjlwa_1,\dots,\adjlwa_b\ge 0\}\in \sigma(\lws_{n-b+1},\dots,\lws_n).
\end{align*}
Observe that $\adjlw$ has the same law as~$\lw$, so $\Omega^+_n$ and $\bar{\Omega}^+_n$ are
independent and have equal probability. Moreover, by Theorem 1 and 2 in~\cite{Vysotsky2010},
$\Prb(\Omega^+_n)=\Theta(n^{-1/4})$. 

Furthermore, on the event $\{\lw_n=\lwa_n=0\}$, we see that $\adjlwa_k=\lwa_{n-k}$ for any $k\in [n]$, so 
\begin{align*}
 \Prb\left(\lw_n=\lwa_n=0,\,\lwa_1,\dots,\lwa_n\ge 0\right)
 &\le \Prb\left(\Omega^+_n\cap \bar\Omega^+_n\cap \{\lw_n=\lwa_n=0\}\right)\\
 &=\Prb(\Omega^+_n)^2\Prb\left(\lw_n=\lwa_n=0 \mid \Omega^+_n\cap\bar\Omega^+_n\right)\\
 &=\Theta(n^{-1/2})\Prb\left(\lw_n=\lwa_n=0 \mid \Omega^+_n\cap\bar\Omega^+_n\right).
\end{align*}

Now, we observe that $\Omega^+_n\cap\bar\Omega^+_n$ only depends on $\lws_1,\dots,\lws_b$
and $\lws_{n-b+1},\dots,\lws_n$, so 
\begin{multline*}
 \Prb\big(\lw_n=\lwa_n=0 \mid \Omega^+_n\cap\bar\Omega^+_n\big)\\
 \le\sup_{(\ell_i)}
 \Prb\big(\lw_n=\lwa_n=0 \mid \lws_i=\ell_i\text{ for }i\in\{1,\dots,b\}\cup\{n-b+1,\dots,n\}\big),
\end{multline*}
where the supremum is over all choices of $\ell_1,\dots,\ell_b,\ell_{n-b+1},\dots,\ell_n\in\{-1,0,1\}$.
We see that given the values of $\lws_1,\dots,\lws_b$ and $\lws_{n-b+1},\dots,\lws_n$, the processes
$(\lw,\lwa)$ restricted to $b+1,\dots,n-b$ have the same distribution as a lazy random walk and its
integrated counterpart with both processes started at some different point, which implies that
\[
 \Prb\big(\lw_n=\lwa_n=0 \mid \Omega^+_n\cap\bar\Omega^+_n\big)
 \le \sup_{y,a}\Prb\big(\lw_{n-2b}=y,\,\lwa_{n-2b}=a\big),
\]
which is $\Theta(n^{-2})$ by Lemma~\ref{lem:locallimitlazy}. The result now follows.

\end{document}